%% file: main.tex
\documentclass[11pt,reqno]{amsart}       

\usepackage{amsmath,amsthm,amsfonts,amssymb,mathtools,dsfont,psfrag,mathrsfs,comment}
\usepackage{enumitem,linegoal}
\usepackage{tikz-cd}
\usepackage{hyperref}
\usepackage{mathabx}
\usepackage[all]{xy}
\usepackage{graphicx}
\usepackage{float}
\hypersetup{
     colorlinks=true,
     linkcolor=red,
     filecolor=blue,
     citecolor = blue,      
     urlcolor=blue,
     }
\usepackage{cleveref}
\usepackage[top=1.5in, bottom=1.5in, left=0.7in, right=0.7in]{geometry}  

\theoremstyle{plain}
\newtheorem{theorem}{Theorem}[section]
\crefname{thm}{theorem}{theorems}

\newtheorem{thmintro}{Theorem}

\crefname{thmintro}{theorem}{theorems}

\newtheorem{definition}[theorem]{Definition}

\crefname{corintro}{corollary}{corollaries}
\newtheorem{lemma}[theorem]{Lemma}
\newtheorem{corollary}[theorem]{Corollary}
\newtheorem{proposition}[theorem]{Proposition}
\newtheorem{example}[theorem]{Example}

\theoremstyle{definition}
\newtheorem{remark}[theorem]{Remark}
\newtheorem*{remark*}{Remark}

\newtheorem*{question*}{Question}

\newcommand{\D}{\mathbb{D}}
\newcommand{\Q}{\mathbb{Q}}
\newcommand{\Z}{\mathbb{Z}}
\newcommand{\ind}{\mathrm{ind}}

\def\la{\langle\,}
\def\ra{\,\rangle}

\newcommand{\bigslant}[2]{{\raisebox{.5em}{$#1$}\left/\raisebox{-.5em}{$#2$}\right.}}

\def\co{\colon\thinspace}
\def\coeq{\coloneqq\thinspace}

\newcommand{\RR}{\mathbb{R}}
\newcommand{\CC}{\mathbb{C}}
\newcommand{\TT}{\mathbb{T}}
\newcommand{\NN}{\mathbb{N}}

\newcommand{\DD}{\mathbb{D}}
\newcommand{\QQ}{\mathbb{Q}}
\newcommand{\SSS}{\mathbb{S}}

\DeclareMathOperator{\image}{image}

\DeclareMathOperator{\Int}{int}

\DeclareMathOperator{\Id}{Id}

\DeclareMathOperator{\CHA}{CH}

\def\rd{{\mathrm d}}

\newcommand{\neigh}{\mathcal{N}}

\newcommand{\norm}[1]{\left\vert#1\right\vert}

\newcommand{\calN}{\mathcal{N}}

\def\cM{{\mathcal M}}

\DeclareMathOperator{\lcm}{lcm}

\begin{document}

\title{
Tight contact structures without symplectic fillings are everywhere}


\author[J.~Bowden]{Jonathan Bowden}
\address{Leibniz Universit\"at Hannover, Germany}
\email{jonathan.bowden@math.uni-hannover.de}
\author[F.~Gironella]{Fabio Gironella}
\address{CNRS - Laboratoire de Mathématiques Jean Leray, Nantes Université, France}
\email{fabio.gironella@cnrs.fr}
\author[A.~Moreno]{Agustin Moreno}
\address{Heidelberg Universit\"at, Germany}
\email{agustin.moreno2191@gmail.com}
\author[Z.~Zhou]{Zhengyi Zhou}
\address{State Key Laboratory of Mathematical Sciences, AMSS, CAS, China \newline \indent
Morningside Center of Mathematics \& Institute of Mathematics, AMSS, CAS, China}
\email{zhyzhou@amss.ac.cn}

\date{}

\begin{abstract}
\noindent
 We show that for all $n \ge 3$, any $(2n+1)$-dimensional manifold that admits a tight contact structure, also admits a tight but non-fillable contact structure, in the same almost contact class. For $n=2$, we obtain the same result, provided that the first Chern class vanishes.
We further construct Liouville but not Weinstein fillable contact structures on any Weinstein fillable contact manifold of dimension at least $7$ with torsion first Chern class.
\end{abstract}

\maketitle

\input{s1}

\input{s2}

\input{s3}

\input{s4_detail}

\input{s5}
\input{s6}

\bibliographystyle{alpha}      
\bibliography{biblio}


\end{document}

%% file: s1.tex
\section{Introduction}

In contact topology, there is an important dichotomy dividing contact structures into those that are \emph{overtwisted} and those that are \emph{tight}. 
 According to \cite{Eli89,BEM}, the former are in fact topological and not geometric objects, in the sense that they satisfy $h$-principles so that their existence and classification can be reduced to classical obstruction theory and algebraic topology. 
 On the other hand, tight contact structures are inherently geometric, hence in general much harder to construct, and let alone classify. 
There is, however, an important sufficient condition for tightness: every contact manifold which is \emph{symplectically fillable} is necessarily tight \cite{Gro85,Eli90,Nie06,BEM}. Here we consider \emph{strong} fillings, i.e.\ symplectic manifolds such that the contact structure at the boundary is defined by the contraction of the symplectic form with a Liouville vector field, locally defined near the boundary.
A natural question is then to understand the difference between the classes of tight and strongly fillable contact structures. Our first result addresses this question in broad generality, on manifolds of dimension at least five.

\begin{thmintro}\label{thm:tightnonfill_general}
If $(M^{2n+1},\xi)$ is tight with $n \ge 3$, then $M$ admits a tight, non-strongly-fillable contact structure in the same almost contact class. 
If $n =2$, the same holds, under the additional assumption that the first Chern class of $\xi$ is zero. 
\end{thmintro}
\noindent 

The assumption on the dimension in \Cref{thm:tightnonfill_general} is necessary, since in dimension three the statement is false. For example, the sphere $\mathbb S^3$ admits a unique tight contact structure \cite{Eli92}, which is also symplectically fillable by the standard Darboux ball. Historically the first examples of tight contact structures that are not symplectically fillable were found in dimension $3$ by Entyre-Honda \cite{EtnHon} and on a special class of product manifolds in dimensions 5 and greater by Massot-Niederkr\"uger-Wendl \cite{MNW}.

One can view \Cref{thm:tightnonfill_general} as suggesting that contact topology in higher dimensions exhibits significantly more flexibility than in dimension three and all contact phenomena ought to occur independent of the underlying smooth topology. In particular, the proof of \Cref{thm:tightnonfill_general} involves a reduction to the case of a standard smooth sphere, i.e.\ we first prove the following special case.
\begin{thmintro}
\label{thm:tightnonfill}
For every $n \ge 2$, the sphere $\mathbb S^{2n+1}$ admits a tight, non-strongly-fillable contact structure that is homotopically standard.
\end{thmintro}
\noindent Here, \emph{homotopically standard} means homotopic to the standard contact structure among \emph{almost contact structures} (the algebraic topological object underlying contact structures, in terms of which the $h$-principles in \cite{Eli89,BEM} are formulated). We then deduce \Cref{thm:tightnonfill_general} by considering connected sums with these non-fillable examples on spheres to preserve the smooth topology but alter the properties of the contact structure. Of course, this requires some significant effort since both the obstruction to fillability as well as tightness do not behave well under connected sums {\em a priori}. 

\subsection*{Hierarchy of fillings} The notion of fillability comes with several different nuances, namely, one has notions of \emph{Weinstein fillability}, \emph{Liouville fillability} and \emph{strong fillability}. 
Thus ordered, each notion is stronger than the following one, and all of them in fact imply tightness. One thus has the following natural \emph{hierarchy} of contact manifolds:
$$
\{\mbox{Weinstein fillable}\}\subseteq \{\mbox{Liouville fillable}\}
\subseteq \{\mbox{strongly fillable}\}\subseteq  \{\mbox{tight}\}.
$$
There has been a considerable amount of work over almost 30 years in studying these inclusions, and they are known to be strict 
in any odd dimension $\ge 3$ \cite{Eli96,EtnHon,Ghi05,Gay,NieWen11, Bow_exact, MNW, BCS1, Zhou_2021, BGM, GhiNie22}. On the other hand, all previously known examples of contact structures with exotic fillability properties come from special geometric constructions, all of which again require the presence of non-trivial topology in the underlying smooth manifold. Now, as a consequence of \Cref{thm:tightnonfill_general} and \Cref{thm:tightnonfill}, the final inclusion is strict for all manifold of dimension $\ge 7$, as soon as this admits at least one tight contact structure.


\begin{remark*}
In fact, there is also the further notion of {\em weak fillability}, which, as the name suggests, is the weakest possible notion of fillability. 
Since this in fact agrees with strong fillability on spheres due to \cite[Proposition 6]{MNW}, we do not consider weak fillability here. 
In particular, even though strictly speaking in (the proof of) \Cref{thm:tightnonfill} we obstruct strong fillability, this equivalence of strong and weak fillability for spheres implies that those contact spheres in fact are also not weakly fillable.
\end{remark*}

Our next result addresses the strictness of the first inclusion in the hierarchy above by constructing examples that are Liouville fillable but admit no Weinstein filling on spheres of sufficiently large dimension.
\begin{thmintro}
    \label{thm:exact_non_Stein_sphere}
For any $n \ge 3$ there exist homotopically standard Liouville fillable contact structures on $\mathbb S^{2n+1}$ that are not Weinstein fillable.
\end{thmintro}
By taking connected sum with examples on spheres, we can further deduce that there are abundant Liouville fillable contact manifolds without Weinstein fillings in the case that the first Chern class is torsion (cf.\ \Cref{thm:exact_non_Stein_general} below).
Thus, other than addressing the $5$-dimensional case, it remains to determine the existence of strongly but not Liouville fillable contact structures on $\mathbb S^{2n+1}$, for all $n\geq 2$.

\subsection*{Infinite classes of examples}
After one constructs examples of contact structures with exotic fillability properties, it is natural to try to produce \emph{infinitely many} such examples as well. 
In this direction, by pushing the techniques used for the proofs of the previous statements further, namely by taking contact connected sums with some other judiciously chosen Weinstein fillable contact structures on the sphere, we obtain the following.

\begin{thmintro}\label{thm:tightnonfill_general_infinite}
Let $n\geq 5$, and suppose that $M^{2n+1}$ admits 
a Weinstein fillable contact structure with torsion first Chern class. 
Then, $M$ admits infinitely many non-isomorphic, tight, and not strongly fillable contact structures in the same almost contact class.
\end{thmintro}

In particular, spheres of dimension at least $11$ admit \emph{infinitely} many tight non-fillable contact structures which are homotopically standard. One also has a similar result for Liouville, but not Weinstein fillable contact structures.

\begin{thmintro}\label{thm:exact_non_Stein_general}
Suppose that $M^{2n+1}$ is of dimension $2n+1 \ge 7$ and admits 
a Weinstein fillable contact structure with torsion first Chern class. 
Then $M$ also admits infinitely many non-isomorphic Liouville fillable contact structures that are not Weinstein fillable in the same formal class.
\end{thmintro}

Again, as a corollary, we obtain that spheres of dimension at least $7$ admit \emph{infinitely} many Liouville but not Weinstein fillable contact structures which are homotopically standard.

\subsection*{Overview of proofs} 
To prove \Cref{thm:tightnonfill}, i.e.\ to obtain a tight and non-fillable contact structure on $\mathbb S^{2n+1}$, we give a geometric construction, as follows. 
We consider the Bourgeois construction on a sphere $\mathbb{S}^{2n-1}$ (with respect to a suitable open book), which yields a contact structure on $\mathbb{S}^{2n-1}\times \mathbb{T}^2$, and then we kill the topology coming from the $\mathbb{T}^2$-factor via isotropic and flexible surgeries to obtain a contact $2n+1$ dimensional sphere. 
The key point, which we now discuss briefly, is to ensure that tightness and non-fillability of the starting Bourgeois contact manifold are preserved under the contact surgeries.

\smallskip
\textit{Tightness.} In general, tightness is known to be preserved under surgery only in dimension $3$ by a result of Wand \cite{MR3418529}.
Here, we circumvent the lack of an analogous result in higher dimensions by using the stronger notion of \emph{algebraic tightness}, i.e.\ non-vanishing of contact homology, and the fact that this is known to be preserved under contact surgeries and to imply tightness (see \Cref{prop:property}). 
Therefore, any contact manifold obtained via contact surgery from an algebraically tight one is in particular (algebraically) tight.

In our situation, it hence suffices to arrange that the starting Bourgeois contact structure is algebraically tight, which can be done as follows. 
When $n=2$, we follow the same strategy in \cite{BGM}, leveraging the existence of a certain exact cobordism with hypertight concave end and convex end, the desired Bourgeois manifold, and appealing to the functoriality of contact homology. When $n\ge 3$, we can make the Bourgeois construction satisfy certain index-positivity condition (more precisely, one can arrange them to be {\em $1$-asymptotically dynamically convex}, as introduced in \cite{Laz20,Zho21}; see \Cref{lem:kADC_stable}), which implies algebraic tightness.

\smallskip

\textit{Non-fillability.} The obstruction to fillability is derived from restrictions of fillings for certain flexibly Weinstein fillable contact manifolds. Taking the celebrated Eliashberg--Floer--McDuff theorem \cite{McD91} as a starting point, many contact manifolds are expected to have unique Liouville/symplectically aspherical fillings, at least up to diffeomorphism. 
Up until now, results in this direction mainly apply to the contact boundary of a split manifold of the form $V\times \D^2$, as well as those obtained via subcritical/flexible surgeries; cf.\ \cite{OV,BGZ,Zho21,product,BGM,GKZ}. If one considers the more general class of strong fillings, these are no longer unique, as one can take a blow-up of any given filling without affecting the boundary. 

However, it is plausible that all the ambiguity of fillings of $\partial(V\times \D^2)$ only comes from birational surgeries like blow-ups, which will increase the complexity of the topology of the filling; see \cite[\S 8]{Zho21} for supporting evidence. 
Moreover, Eliashberg \cite{Eli90} showed in dimension $3$, that if $W$ is a symplectic filling of $Y\#Y'$, then $W$ is obtained from attaching a $1$-handle to a symplectic filling $W'$ of $Y \sqcup Y'$ (which might be connected, cf.\ \cite{Bow_exact}). Although the higher-dimensional analogue has not been established, partial generalizations were obtained by Ghiggini-Niederkr\"{u}ger-Wendl \cite{GNW}. Incarnations of these expectations at the level of homology lead to \Cref{thm:unique}, putting homology constraints on putative strong fillings, which serve as our obstructions to the existence of strong fillings. 

In \cite{BGM}, the first three authors used moduli spaces of closed spheres that appear after applying a certain capping procedure, similar as to what is done in \cite{MNW} to obstruct fillability of certain Bourgeois contact manifolds. In this paper, we instead use a partial cap (described in \Cref{prop:cap_gen}) to build a strong cobordism from the  Bourgeois contact manifold as well as the exotic sphere to a flexibly fillable manifold, for which we can apply \Cref{thm:unique}. The full proof does require that symplectic cohomology can be defined for general strong fillings, which has now been established, for instance, using Pardon's VFC package \cite{pardon2019contact}.  

\smallskip

For the proof of \Cref{thm:tightnonfill_general}, one simply takes the original tight contact manifold, and performs a connected sum with the exotic sphere of \Cref{thm:tightnonfill}.
The subtle point is to ensure that tightness and non-fillability are preserved under this connected sum. 
The proof of non-fillability uses the fact that the homological obstructions can be made to persist under connected sum, as explained above. 
Tightness is again reduced to algebraic tightness: one can assume without loss of generality that the original contact manifold is strongly fillable (otherwise there is nothing to prove), and therefore it is algebraically tight; then, we can appeal to the fact (see \Cref{prop:property}) that algebraic tightness is preserved under connected sum, that leads in particular to tightness of the considered connected sum, as desired.

\smallskip

The proof of \Cref{thm:exact_non_Stein_sphere} yielding obstructions to Weinstein fillings is technically simpler. The construction is the following: we start from a Liouville domain of the form $V^{2n}=N\times [-1,1]$ as constructed by \cite{MNW}. We let $M=\partial(V\times \mathbb D^2)$, which admits a flexible Weinstein cobordism to a homotopy sphere by \Cref{lem:surg_parallel}. Moreover, up to taking finitely many self-connected sums, it can be assumed to be the standard smooth sphere. The cohomology of the filling resulting from stacking the cobordism on top of $V\times \mathbb D^2$ is non-trivial in degree $2n-1$, which evaluates non-trivially on the fundamental class of $N$. By results in \cite{Zho21}, the cohomology of any Weinstein filling of this contact sphere is the same as the cohomology of the natural Liouville one, so in particular non-trivial in degree $2n-1>n+1$ (as $n\ge 3$), which obstructs the Weinstein property. Therefore, we obtain a Liouville but non-Weinstein fillable contact structure. The standard almost contact structure is also achieved by a connected sum trick.

\medskip

Finally, the examples in \Cref{thm:tightnonfill_general_infinite} are also obtained by considering connected sums of the examples from \Cref{thm:tightnonfill_general} with some flexibly filled examples that are constructed in a somewhat ad-hoc manner.
Tightness and non-fillability follow from the same line as for \Cref{thm:tightnonfill_general}, and the tool used to distinguish them up to contactomorphism is positive symplectic cohomology, which is a contact invariant in this setting thanks to $1$-asymptotic dynamical convexity and can be defined without using a Liouville filling by \cite{CieOan18}.
Similarly, \Cref{thm:exact_non_Stein_general} follows from connect-summing the examples of \Cref{thm:exact_non_Stein_sphere} with infinitely many homotopically trivial flexibly fillable contact spheres from \cite{Laz20}, and they are distinguished using again positive symplectic cohomology but this time of their natural Liouville fillings, that is an invariant of the contact structure by \cite{Laz20}.

\subsection*{Outline of Paper} \Cref{sec:Bourgeois} contains important preliminaries and background, including a description of the Bourgeois construction together with a result concerning dynamical convexity of its Reeb vector field, as well as statements about the existence of strong cobordisms from contact manifolds, obtained from subcritical (resp.\ flexible) surgeries on $\SSS^1-$equivariant ones, to subcritically (resp.\ flexibly) fillable ones.

In \Cref{sec:alg_tight}, we describe the notion of algebraic tightness, together with relevant properties, and prove that the Bourgeois construction is a rich source of contact manifolds satisfying this property. 

\Cref{sec:obstr_fillings} contains results on obstructions to strong fillability via symplectic cohomology and \Cref{sec:geometric_construction} contains the proof of \Cref{thm:tightnonfill} and  \Cref{thm:tightnonfill_general,thm:tightnonfill_general_infinite}. 
Lastly, \Cref{sec:exact_non-Weinstein} contains the proofs of \Cref{thm:exact_non_Stein_sphere,thm:exact_non_Stein_general}.

\subsection*{Acknowledgements}
The authors are grateful to Peter Albers and Yasha Eliashberg, for productive discussions and interest in the project,  to Fr\'ed\'eric Bourgeois, Roger Casals, Otto van Koert, Oleg Lazarev, Alex Oancea, for helpful feedback on an earlier version of this paper, and to anonymous referees for many suggestions that improved this paper.

J. Bowden gratefully acknowledges the support from the Heisenberg-Grant (BO 4423/4-1) of the German Research Council.
F.\ Gironella is supported by the region Pays de la Loire, via the
project Étoile Montante 2023 SymFol, and gratefully acknowledges the support of the Sonderforschungsbereich Higher Invariants 1085.
A.\ Moreno received support from the National Science Foundation under Grant No.\ DMS-1926686, and is currently supported by the Sonderforschungsbereich TRR 191 Symplectic Structures in Geometry, Algebra and Dynamics, funded by the DFG (Projektnummer 281071066 – TRR 191), and also by the DFG under Germany's Excellence Strategy EXC 2181/1 - 390900948 (the Heidelberg STRUCTURES Excellence Cluster). 
Z.\ Zhou is supported by National Key R\&D Program of China under Grant No.2023YFA1010500 and National Natural Science Foundation of China under Grant No.\ 12288201 and 12231010. 
This project has also received funding from the European Research Council (ERC) under the European Union’s Horizon 2020 research and innovation programme (grant agreement No.\ 772479), as well as from the Agence Nationale de la Recherche (ANR) under the Project COSY (ANR-21-CE40-0002-04).

%% file: s2.tex
\section{Bourgeois contact manifolds}\label{sec:Bourgeois} 

In this section, we recall the Bourgeois construction \cite{Bo}, following the presentation that was given in \cite{BGM}. 

Consider a closed, oriented, connected smooth manifold $M^{2n-1}$ and an open book decomposition that we denote $(B,\theta)$, together with a defining map $\Phi\co M\rightarrow \RR^2$ such that each $z\in\Int(\DD^2)$ is a regular value. 
Here, $B\subset M$ is a closed codimension-2 submanifold, $\theta \colon M \backslash B \rightarrow \mathbb{S}^1$ is a fiber bundle, and $\Phi$ is such that $\Phi^{-1}(0)=B$ and $\theta=\Phi/\norm{\Phi}$.
Let us also denote by $\rho$ the norm $\vert \Phi\vert$.

Recall that a $1$-form $\alpha$ on $M$ is said to be \textit{adapted} to $\Phi$ if it induces a contact structure on the regular fibers of $\Phi$ and if $d\alpha$ is symplectic on the fibers of $\theta=\Phi/\norm{\Phi}$.
In particular, if $\xi$ is a contact structure on $M$ supported by $(B,\theta)$, in the sense of \cite{Gir}, then there is such a pair $(\alpha,\Phi)$ with $\alpha$ defining $\xi$ as follows from the definition.

\begin{theorem}[Bourgeois \cite{Bo}]
	\label{thm:bourgeois}
	Let $(B,\theta)$ denote an open book decomposition of $M^{2n-1}$, represented by a map $\Phi=(\Phi_1,\Phi_2)\co M \rightarrow \RR^2$ as above, and let $\alpha$ be a $1$-form adapted to $\Phi$.
	Then, $\alpha_{BO} \coeq \alpha + \Phi_1 d q_1 - \Phi_2 dq_2$ is a contact form on $ M\times \TT^2$, where $(q_1,q_2)$ are coordinates on $\TT^2$.
\end{theorem}

The contact form $\alpha_{BO}$ on $M\times\TT^2$ will be called the \textit{Bourgeois form} associated to $(\alpha,\Phi)$ in what follows. It will be useful to also think in terms of abstract open books, i.e.\ in terms of the Liouville page $\Sigma = (\Sigma,d\lambda)$ and the compactly supported symplectic monodromy $\psi$. 
In this case, we write $OB(\Sigma,\psi)$ for the contact manifold obtained via the Thurston-Winkelnkemper construction, and we denote $BO(\Sigma,\psi)=(M\times \mathbb{T}^2,\ker \alpha_{BO})$.
We invite the reader to consult \cite{LMN} for details on the correspondence between Bourgeois contact structures for geometric and abstract open books.

\begin{remark}
\label{rem:almost_cont}
The almost contact structure underlying the Bourgeois contact manifold is (up to homotopy) just the sum 
$$\left(\xi,d\alpha\right) \oplus (T\TT^2,\Omega),$$
where $\Omega$ is an area form on $\mathbb{T}^2$. 
This can be seen via an explicit homotopy; see, for instance \cite[Lemma 4.1.(1)]{Gir20}. 
In particular, if $c_1(\xi)$ is torsion, then the same is true for the Bourgeois contact structure, independently of the auxiliary choice of open book.
\end{remark}

\subsection{Reeb dynamics of Bourgeois contact structures} 
\label{sec:reeb_dynamics}

In this section, we consider the Reeb dynamics of the Bourgeois contact forms more closely. 
First of all applying a lemma of Giroux (see e.g.\ \cite[Section 3]{DGZ14}), on a neighborhood of the form $\calN = B \times \DD^2\times  \TT^2$ of $B$ where $\rho=\vert \Phi\vert$ coincides with the radial coordinate of $\DD^2$, the Bourgeois contact form looks like 
\begin{align*}
\alpha_{BO}&= \alpha +\rho(r)(\cos(\theta)\rd q_1- \sin(\theta)\rd q_2)\\
&= h_1(r)\alpha_B + h_2(r)\rd \theta+ r(\cos(\theta)\rd q_1- \sin(\theta)\rd q_2).
\end{align*}
Here, $\alpha_B=\alpha\vert_B$ is the contact form on the binding, $(r,\theta)$ are polar coordinates on $\DD^2$, and the functions $h_1,h_2$ satisfy the following conditions:
\begin{enumerate}[label=\roman*.]
		\item $h_1(0)>0$ and $h_1(r)=h_1(0) - r^2$ near $r=0$;
		\item  $h_2(r)\sim r^2$ for $r\rightarrow0$;
		\item  $\frac{h_1^{n-1}}{r}\,(h_1h_2'-h_2h_1')>0$  for $r\geq 0$ (contact condition);
		\item $h_1'(r)<0$ for $r>0$ (symplectic condition on the pages). 
\end{enumerate}

\noindent
We also point out that there is a natural (orientation-preserving) diffeomorphism
\begin{equation}
\label{eqn:isomorph_cotangent}
    \begin{split}
        \calN = B\times \DD^2\times\TT^2 &\to B \times D^*\TT^2 \\
        (b,p_1,p_2,q_1,q_2)&\mapsto (b,p_1,q_1,-p_2,q_2).
    \end{split}
\end{equation}

We have the following global description of the Reeb vector field (c.f.\ \cite{Bo,Gir20}) on a Bourgeois manifold.
\begin{lemma}
	\label{lemma:Reeb_vf}
	The Reeb vector field of the contact form $\alpha_{BO}$ is given globally by
	\begin{equation}\label{eqn:reeb_vf}
	R_{BO}=\mu(r)R_B + \nu(r) [\cos(\theta)\partial_{q_1}- \sin(\theta)\partial_{q_2}],
	\end{equation}
	where 	
	$R_B$ is the Reeb vector field of the restriction $\alpha_B=\alpha\vert_B$, and the coefficients are
	\begin{equation*} \mu \, =\, \begin{cases} \frac{\rho'}{\rho'h_1-\rho h_1'} & \text{in } \neigh \\  0 & \text{elsewhere } \end{cases} 
	\, \, \text{ and }\,\, 
	\nu \, = \begin{cases} \frac{-h_1'}{\rho' h_1-\rho h_1'} & \text{in } \neigh \\  1 & \text{elsewhere } \end{cases} \text{ .}
	\end{equation*}
\end{lemma}

\subsection{ADC Bourgeois contact manifolds}\label{subsec_ADC}
In this section, we will assume that the first Chern class $c_1(\xi)$ of all contact structures is \emph{torsion}. In this case, there is a well-defined (rational) Conley-Zehnder index $\mu_{CZ}(\gamma)$ for any non-degenerate contractible periodic Reeb orbit $\gamma$. If the contact manifold is of dimension $2n-1$, then the \textit{SFT degree} of $\gamma$ is defined as $$|\gamma| \coloneqq n-3 + \mu_{CZ}(\gamma).$$ Notice that this does \textit{not} correspond to the degree in ($S^1$-equivariant) symplectic cohomology of $\gamma$ when viewed as a periodic Hamiltonian orbit, which is just given by $n-\mu_{CZ}$.

We now recall the following definition from \cite{Zho21}, which generalizes the one from \cite{Laz20}. In what follows, if $\alpha_1,\alpha_2$ are contact forms on a manifold $M$, we write $\alpha_1>\alpha_2$ if there exists a smooth function $f:M\rightarrow \mathbb R$ which satisfies $f>1$, such that $\alpha_1=f\alpha_2$.
\begin{definition}
\label{def:kADC}
    A contact structure $(M,\xi)$ is called $k$ asymptotically dynamically convex ($k$-ADC) if there is a sequence of contact forms $\alpha_1 > \alpha_2 > \cdots > \alpha_i > \dots  $, and a sequence of real numbers $D_i$ with $D_i \to \infty$, so that all \textbf{contractible} periodic orbits of $\alpha_i-$action $< D_i$ are non-degenerate and have degree\footnote{From now on, unless specified otherwise, for simplicity \emph{degree will mean the SFT degree}.} $> k$.
\end{definition}

In particular, $0$-ADC as in \cite{Zho21} is just ADC as in \cite{Laz20}. 
An important special case is when there is a contact form $\alpha$ that is \textit{index-positive}, meaning, following \cite[Section 9.5]{CieOan18}, that all contractible periodic orbits are non-degenerate and have positive degree for $\alpha_i \equiv \alpha$ the given contact form. Here, we introduce a slight generalization of index-positivity as follows:

\begin{definition}\label{def:index-pos}
    A contact form $\alpha$ on a contact manifold $(Y,\xi)$ with $c_1(\xi)$-torsion, is called \emph{$k$-index positive}, if the contact form $\alpha$ is Morse-Bott non-degenerate along contractible Reeb orbits and $|\gamma|:=n-3+\mu_{LCZ}(\gamma)$
    is strictly bigger than $k$ for all contractible Reeb orbits $\gamma$, where $\mu_{LCZ}$ is the lower semicontinous extension of the Conley-Zehnder index. 
\end{definition}
We recall that the degree $\vert \cdot \vert$ in \Cref{def:index-pos} is known as the \emph{lower SFT degree}, as introduced by McLean \cite{Reeb}. Recall also from \cite[Definition 1.7]{Bothesis} that the Reeb flow $\varphi_t$ of a contact form $\alpha$ is called \emph{Morse-Bott non-degenerate} if closed periodic orbits with a given action $T$ come in closed smooth submanifolds $N_T$ so that the rank of $\rd\alpha\vert_{N_T}$ is locally constant and $TN_T=\ker(\rd\varphi_T-\Id)$ at all points of $N_T$. 

We also point out that the index $\mu_{LCZ}$ is nothing else than the lowest Conley-Zehnder index of the Reeb orbits obtained from the Morse-Bott family after breaking the Morse-Bott symmetry using the perturbations in \cite[Lemma 2.3]{Bothesis}. 
Indeed, $\mu_{LCZ}(\gamma)$ equals $\mu_{RS}(\gamma)-\frac{1}{2}(\dim N_T-1)$ by \cite[Lemma 2.4]{Bothesis}, where $\mu_{RS}$ is the Robbin-Salamon index \cite{RS} (it is called generalized Maslov index in \cite[Section 5.2.2]{Bothesis}, see \cite{Gutt} for a thorough discussion of all these indices); see also \cite[\S 2]{Reeb}. 
By \cite[Lemmas 2.3 and 2.4]{Bothesis}, $k$-index positive contact manifolds are $k$-ADC, as one can perturb the $k$-index positive Morse-Bott contact form into a shrinking sequence of contact forms satisfying the conditions in the definition of $k$-ADC.

A very useful aspect of the ADC condition is that it is preserved under flexible surgery \cite{Laz20}, whereas index-positivity may not be.  
Lastly, we point out that the original definition of ADC in \cite{Laz20} used a \emph{non-increasing} sequence of contact forms; however, one can always find a \emph{strictly decreasing} sequence, as explained in \cite[Remark 3.7.(4)]{Laz20}.

\begin{example}[Brieskorn Manifolds]
\label{ex:brieskorn}
Consider the $(2n-1)$-dimensional Brieskorn spheres $\Sigma_n(k,2, \cdots, 2)$ given as the link of the $A_{k-1}$-singularity, i.e.\
\[
\Sigma_n(k,2,\cdots, 2)=\{z_0^k +z_1^2 + \cdots+ z_n^2 = 0  \}\cap S^{2n+1}\subset \CC^{n+1} .
\]
This is an index-positive contact manifold. 
The degrees of the generators are at least $2n-4$, provided that $n \ge 3$. 
This follows from \cite{U99} in case $n=2m+1$ is odd, and from \cite[Section 3]{vKo08} in general.

More generally, one can consider Brieskorn spheres $$\Sigma_n(p_1,p_2, \cdots, p_{n-2},2,2,2)$$ for odd primes $p_i$, which according to \cite[Section 4.1]{vKo08}, is index positive in the Morse-Bott case as in \Cref{def:index-pos}.
\end{example}

\begin{example}[3-dimensional case]
\label{ex:brieskorn_3d}
For the case of $3$-dimensional Brieskorn manifolds, the computations are more involved as the manifolds are not simply connected. One can view links of $A_{k-1}$-singularities as quotients of the standard contact $3$-sphere so that contractible periodic orbits on the links can be lifted to the $3$-sphere. Those contractible orbits then have Conley-Zehnder indices at least $3$. The corresponding contact form on the link is Morse-Bott as it is the quotient of the round contact form on $S^3$, which is Morse-Bott. However,  when we construct a $7$-dimensional Bourgeois contact structure $\SSS^5\times \TT^2$ with the $3$-dimensional $A_{k-1}$-singularity link as the binding, these non-contractible orbits in the binding become contractible in the Bourgeois contact manifold. Since $H_1(\Sigma(k,2,2);\Z)$ is torsion, all Reeb orbits, including those non-contractible orbits, can be assigned well-defined rational Conley-Zehnder indices \cite[\S 4]{Reeb}, which are at least $1$ by the computation in \cite[\S 2.1]{McKay}. In view of this, the 3-dimensional links of $A_{k-1}$-singularities are $(-1)$-index positive.
\end{example}

We next observe that index-positivity is stable under the Bourgeois construction in the following sense:

\begin{lemma}
\label{lem:kADC_stable}
Assume that $\alpha$ is a contact form adapted to an open book decomposition with defining map $\Phi\colon M \to \RR^2$ on $M$, and that its restriction $\alpha_B$ to the binding $B$ is $k$-index positive and $B$ is simply connected.
Then, the Bourgeois form $\alpha_{BO}$ constructed as in \Cref{thm:bourgeois} is $(k+2)$-index positive.

In particular, the associated Bourgeois contact structure on $M\times \TT^2$ is $(k+2)$-ADC.
\end{lemma}

\begin{proof}
An explicit computation using the expression in \Cref{lemma:Reeb_vf} shows that, on the neighborhood $\calN=B\times \DD^2 \times \TT^2$, the flow $\psi^t$ of $R_{BO}$ at time $t>0$ starting at a point $(b_0,r_0,\theta_0,q_{1,0},q_{2,0})$ is given by
    \begin{equation*}
        b(t)=\psi_B^{t\mu(r_0)}(b_0) \, , \quad
        r(t)=r_0\, , \quad
        \theta(t)=\theta_0 \, , \quad
        q_1(t)= q_{1,0} + t\nu(r_0)\cos\theta_0 \, , \quad
        q_2(t) = q_{2,0} - t \nu(r_0) \sin\theta_0 \, ,
    \end{equation*}
    where $\psi_B^s$ is the Reeb flow of $\alpha_B$ on $B$ at time $s$.
    Similarly, on the complement of $\calN=B\times \DD^2\times \TT^2$, where $\nu= 1$ and $\mu =0$, the flow at time $t$ is just given, in coordinates $(x,q_1,q_2)\in M \times \TT^2$ and with starting point $(x_0,q_{1,0},q_{2,0})$, by 
    \begin{equation*}
        x(t)=x_0 \, , \quad
        q_1(t)= q_{1,0} + t\cos\theta_0 \, \quad
        q_2(t) = q_{2,0} - t \sin\theta_0 \, .
    \end{equation*}
    
    In particular, qualitatively, we have the following behavior.
    Closed Reeb orbits of $R_{BO}$ that touch $B\times \{0\}\times \TT^2$ are in fact contained in it, and they organize in $\TT^2$-parametric families, each coming from a closed Reeb orbit of $\alpha_B$.
    All the other closed Reeb orbits that lie in the complement of $B\times \{0\} \times \TT^2$ project as a linear closed curve in the $\TT^2$ factor.
    Hence, under our assumption of $B$ simply connected, the only contractible Reeb orbits of $R_{BO}$ are those of the form $\gamma_q=\gamma_B\times \{(0,q)\}$ with $\gamma_B$ a closed Reeb orbit of $\alpha_B$ and $q\in\TT^2$.

    Notice that the hypotheses on the function $h_1$ mean that $\nu(r)=\frac{2r}{h+r^2}$ for $r<1$, i.e.\ on $\calN = B\times \DD^2\times \TT^2$, where $h:=h_1(0)>0$. 
    Then if we use coordinates $(p_1,q_1,p_2,q_2)\in D^*\TT^2\simeq \DD^2\times \TT^2\ni (p_1,-p_2,q_1,q_2)$ as in \eqref{eqn:isomorph_cotangent}, we can rewrite the $q$-coordinates of the above described flow $\psi^t$ as
    $$q_1(t)=q_{1,0}+t\frac{p_{1,0}}{h+r_0^2},\quad q_2(t)=q_{2,0}+t\frac{p_{2,0}}{h+r_0^2}.$$
    It is then clear that the linearized flow along $B\times \{0\} \times \TT^2$ splits in the $B$ direction and $D^* \TT^2$ direction. 
    The component in the $D^* \TT^2$-direction of the linearized flow, that we denote by $\rho(t)$, is given, in the $(p_1,q_1,p_2,q_2)$-coordinates on $D^*\TT^2$, by
    $$\rho(t)=\left[\begin{array}{cccc}
         1 &  0 & 0 &  0\\
         \frac{t}{h} &  1 & 0 & 0 \\
         0 & 0 & 1 & 0\\
         0 & 0 & \frac{t}{h} & 1
    \end{array}\right].$$
    In particular, the linearized return map along $\gamma_q=\gamma_B\times\{(0,q)\}$ has, in the $D^* \TT^2$-direction, as $1$-eigenspace the $\TT^2$-direction (the zero section). 
    Since $\alpha_B$ is Morse-Bott, it follows that $\alpha_{BO}$ is also Morse-Bott along $B\times \{0\} \times \TT^2$. 
    We also have 
        $$\mu_{LCZ}(\gamma_p)=\mu_{LCZ}(\gamma_B)+\mu_{LCZ}(\{\rho(t)\mid 0\le t\le T\})$$
   where $T$ is the period of $\gamma_p$. Since the  Robbin-Salamon index of the shear matrix path $\rho(t)$ is $1$ \cite[Lemma 26 (7)]{Gutt}, we have $\mu_{LCZ}(\{\rho(t)\}_{0\le t\le T})=\mu_{RS}(\{\rho(t)\}_{0\le t\le T})-\frac{1}{2}\dim \ker (\rho(T)-\Id)=0$.  Thus the SFT degree on $M^{2n-1} \times \TT^2$ is  
    \[
    |\gamma_q| = (n+1)-3 + \mu_{LCZ}(\gamma_q) =(n-1)-3 + \mu_{LCZ}(\gamma_B) + 2= |\gamma_B|_{B}+2 \, .
    \]
    As $\vert\gamma_B\vert_{B}>k$ by assumption, we get $\vert\gamma_q\vert>k+2$.
    
    \medskip
    
    Finally, the second conclusion is a direct consequence of the perturbation procedure in \cite{Bothesis} as explained before the lemma.
    %
    %
    This concludes the proof.
\end{proof}

\begin{remark}\label{rmk:torsion}
    When $B$ is not simply connected, there might be Reeb orbits of $B$ which are non-contractible in $B$ but contractible in the ambient open book $OB(\Sigma,\psi)$.
    These orbits should then be included in the verification of the ADC/index-positive condition for the open book, even though they are not relevant for that of $B$. 
    In our applications, this will only happen in settings where $H_1(B;\Z)$ is torsion.
    In this case, if the contact structure on $B$ has torsion first Chern class, \emph{every} Reeb orbit can be assigned a well-defined rational Conley-Zehnder index \cite[\S 4]{Reeb}. In this case, one can then consider the notion of $k$-ADC manifolds but for \emph{all} Reeb orbits. 
    The computation in \Cref{lem:kADC_stable} is still valid for the rational Conley-Zehnder index. 
\end{remark}

We have the following examples that will be used in the proof of \Cref{thm:tightnonfill}.
\begin{example}\label{ex:Brieskorn_ADC}
Consider the Bourgeois contact manifold associated to the Milnor open book on $\SSS^{2n-1}$ coming from the $A_{k-1}$-singularity, whose binding is a Brieskorn sphere $B=\Sigma_{n-1}(k,2,\cdots,2)$ and whose page is Liouville homotopic to the Milnor fiber $V_{n-1}(k,2,\cdots,2)=\{z_0^k+z_1^2+\dots+z_{n-1}^2=\epsilon\}\cap \DD^{2n}$, for small $\epsilon>0$. When $n\ge 4$, $B$ is a simply connected $1$-index positive contact manifold according to \Cref{ex:brieskorn}. 
We thus obtain from \Cref{lem:kADC_stable} that the corresponding Bourgeois contact manifolds on $\mathbb S^{2n-1} \times \TT^2$ are $1$-ADC (and in fact even $3$-ADC, but $1$-ADC will be sufficient for our later purposes). 
When $n=3$, the corresponding Bourgeois contact manifolds on $\mathbb S^{5} \times \TT^2$ are also $1$-ADC by \Cref{ex:brieskorn_3d} and \Cref{rmk:torsion}.
\end{example}

\subsection{$\SSS^1-$invariant contact structures and convex decompositions} 
\label{sec:S1_invariant}

Bourgeois contact structures are $\TT^2$-invariant, and therefore in particular $\SSS^1$-invariant.
As discussed in \cite[Section 6]{DinGei12}, any $\SSS^1-$invariant contact structure on a manifold $V^{2n} \times \SSS^1$ induces a decomposition of the first factor into so-called {\it ideal Liouville domains}, as defined in \cite{GirouxIdealLiouvDom}.
In particular, we have a topological decomposition
\[
V\times \SSS^1 = V_+ \times \SSS^1 \bigcup \overline{V}_- \times \SSS^1.
\]
Any (constant) section $V\times\{pt\}$ is a convex hypersurface, and the decomposition $V = V_+ \cup \overline V_-$ is a convex decomposition.
In the case of the Bourgeois contact manifolds, the ideal Liouville pieces are given as products of the page of the initial open book with an annulus, i.e.\ $V_{\pm} = \Sigma \times \DD^*\SSS^1$. 
We refer to the existing literature for details about what just described: in \cite{GirouxIdealLiouvDom} and \cite[\S 5.2]{MNW} for ideal Liouville domains, in \cite{DinGei12} for $\SSS^1$-invariant contact structures in general, and in \cite[\S 6.1, 6.3]{BGM} for the case of Bourgeois structures as $\SSS^1$-invariant ones.
The pieces in this decomposition are round contactizations of ideal Liouville domains and are referred to as {\it Giroux domains} following \cite{MNW}.

\subsection{Capping Giroux domains} In \cite[Theorem B]{BGM}, certain homological obstructions to the symplectic fillability of Bourgeois contact manifolds were obtained. 
The argument of the proof uses a capping construction of Massot-Niederkr\"uger-Wendl \cite{MNW}. 
More precisely, whenever one finds an embedded Giroux domain of the form $G = X \times \SSS^1$, one can perform a \emph{blow down} construction as in \cite{MNW},
which topologically consists of removing the interior of $G$ and collapsing all circle fibers at the remaining boundary to points. 
By \cite[Section 5.1]{MNW}, the resulting smooth manifold $M_{bd}$ (where $bd$ stands for blown down) carries a contact structure that is well-defined (up to isotopy). 
The following special case of \cite[Theorem 6.1]{MNW} then says that this is achievable through a cobordism:
\begin{theorem}\label{thm:cap_MNW}
Let $M$ be a contact manifold containing a Giroux domain $G = X \times \SSS^1$. Then there is a strong symplectic cobordism from $M$ to the contact manifold $M_{bd}$ obtained by blow down of $G$.
\end{theorem}
\noindent Topologically, this cobordism is obtained by attaching $X\times \DD^2$ on top of $X\times \SSS^1$ (and smoothing corners), where the latter is seen as the positive boundary of $[0,1]\times X\times \SSS^1$.
We emphasize that this cobordism is not Liouville, as there are only local Liouville vector fields near the boundary in general.


In the case where we begin with an $\SSS^1$-invariant contact structure and apply this blow down cobordism to Giroux domain $V_{-} \times \SSS^1$, we end up with a convex boundary $M_{bd}$ which is just the contact open book $OB(V_+,\Id)$ with page $V_+$ and trivial monodromy. 
Equivalently this is then the contact boundary of $V_+ \times \DD^2$. 
In other words, we have the following:
\begin{corollary}\label{cor:cap_inv_case}
Let $M = V \times \SSS^1$ be a contact manifold with an $\SSS^1$-invariant contact structure and convex decomposition
\[
V\times \SSS^1 = V_+ \times \SSS^1 \bigcup \overline{V}_- \times \SSS^1.
\]
Then there is a strong symplectic cobordism from $M$ to the contact boundary of $V_+ \times \DD^2$. Furthermore, if $c_1(\xi)$ is torsion and $V_-$ is connected, then the same is true for the first Chern class of the cobordism. 
\end{corollary}

The last sentence in the above statement follows from the following observation.
The cobordism $W$ above topologically corresponds to attaching the handle $V_- \times \DD^2$ to $M\times [0,1]$ along $\overline{V}_- \times \SSS^1 \subset M \times \{1\}$.
By the long exact sequence in homology associated to the pair $(W,M)$, we have an exact sequence
$$H_2(M;\QQ)\to H_2(W;\QQ) \to H_2(W,M;\QQ) \to H_1(M;\QQ).$$
Since $H_2(W,M;\QQ)\simeq H_2(V_-\times \DD^2,V_-\times \SSS^1;\QQ)\simeq H_2(\{p\} \times \DD^2,\{p\}\times \SSS^1;\QQ)\simeq H_1(\{p\}\times \SSS^1;\QQ)$ when $V_-$ is connected, which injects into $H_1(M=V\times \SSS^1;\QQ)$ for $p\in V_-$. Therefore, the inclusion of the negative end $M \hookrightarrow W$ induces a surjection on $H_2(\cdot;\QQ)$.
In particular, $c_1(W)\in H^2(W;\QQ)$ is determined by its restriction to the negative boundary of $W$, which just coincides with $c_1(\xi)$. 
Hence, the latter being torsion implies that the former is torsion as well.

\subsection{Allowing surgeries on a Giroux domain} 

We now wish to state a generalization of \Cref{cor:cap_inv_case}, where we allow contact surgeries in one part of the convex decomposition. 

\begin{proposition}\label{prop:cap_gen}
Consider an $\SSS^1$-equivariant contact structure on $M = V^{2n} \times \SSS^1$, with induced convex splitting $V =V_+ \cup \overline{V}_-$.
Suppose now that $(M',\xi')$ is obtained from the former via a sequence of contact surgeries on isotropic (possibly Legendrian) spheres contained in the complement of $V_- \times \SSS^1$. 

Then there is a strong symplectic cobordism $W_{cap}$ from $(M',\xi')$ to the contact manifold $(M_+',\xi_+')$ obtained from $OB(V_+,\Id)$ via the corresponding sequence of contact surgeries (i.e.\ those same contact surgeries performed on $V_+\times \SSS^1$ considered as a subset of $OB(V_+,\Id)$). 

Moreover, if $V_+$ is Weinstein, there is a (smooth) copy $V_+'$ of $V_+$ inside $M_+'$, such that the following holds: 
\begin{itemize}
\item $V_+' \subset M_+'$ can be isotoped to lie in the negative end $M'$ of the cobordism $W_{cap}$; 
\item the image of the map induced by the natural inclusion $ H_q(V'_+)  \to H_q(M'_+)$, has dimension at least $\dim H_q(V'_+) - \#(q\text{-surgeries})$.
\end{itemize}
\end{proposition}

\begin{proof}
The existence of the cobordism simply follows from the fact that the blow-down cobordism of \Cref{thm:cap_MNW} can be described as an attachment of a handle $V_-\times \DD^2$ on top of $V_-\times\SSS^1\subset M$, and that the contact surgeries are made in the interior of $V_+\times \SSS^1\subset M$, which is \emph{disjoint} from $V_-\times\SSS^1$; see \Cref{fig:blowdown_handle}.
In particular, in the symplectic cobordism given by blowing-down and performing the contact surgeries, one can find a symplectic sub-cobordism from $M'$ to $M_+'$ as claimed in the statement.

For the second part of the statement, we first point out that the dimension of the ambient manifold is $2n+1$, while $V_+$ is homotopy equivalent to its $n$-skeleton, and the dimension of the attaching spheres for the surgeries is at most $n$.
In particular, up to perturbing the attaching spheres of the surgery handles that give $M_+'$ starting from $OB(V_+,\Id)$ \emph{among isotropic/Legendrian submanifolds}, we can assume that they are disjoint from the $n$-skeleton $V_+^{(n)}$ of a page $V_+$ in the open book by general position. 
In particular,  one naturally has a copy of a neighborhood of $V_+^{(n)}$, which is homotopy equivalent to $V_+$, inside $M_+'$; we denote it $V_+'\simeq V_+$.

What is more, as $M_+'$ is topologically obtained from $M'$ by attaching a handle $V_-\times \DD^2$,  and the subset $V_+' \subset M_+'$ is disjoint from it as well as from the surgery region, then $V_+'$ lives in a region where the cobordism $W_{cap}$ is a trivial product and therefore $V_+'$ can be isotoped down to the negative boundary $M'$ of $W_{cap}$. 

The final claim follows from the fact that $M_+'$ is obtained by performing a certain number, say $k$, of contact surgeries on $OB(V_+,\Id)$, which can at most kill $k$ generators in the homology group of the corresponding degree. 
\end{proof}

\begin{figure}[t] 
	\centering
    \def\svgwidth{0.6\textwidth}
    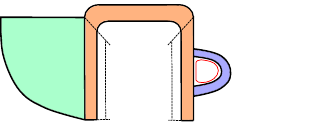
    \caption{The blow-down handle $V_-\times D^2$ in green, attached on top of $\{1\}\times V_-\times \SSS^1 \subset \{1\}\times M$ seen as the boundary of the trivial cobordism $[0,1]\times M$. The purple region depicts the handle attachment that gives $M'$ from $M$, attached on the region $V_+\times \SSS^1\subset M$. In red the contact manifold $M'_+$ resulting from both blow-down and contact surgeries.}
    \label{fig:blowdown_handle}
\end{figure}

%% file: blowdown_handle.pdf_tex
\begingroup%
  \makeatletter%
  \providecommand\color[2][]{%
    \errmessage{(Inkscape) Color is used for the text in Inkscape, but the package 'color.sty' is not loaded}%
    \renewcommand\color[2][]{}%
  }%
  \providecommand\transparent[1]{%
    \errmessage{(Inkscape) Transparency is used (non-zero) for the text in Inkscape, but the package 'transparent.sty' is not loaded}%
    \renewcommand\transparent[1]{}%
  }%
  \providecommand\rotatebox[2]{#2}%
  \newcommand*\fsize{\dimexpr\f@size pt\relax}%
  \newcommand*\lineheight[1]{\fontsize{\fsize}{#1\fsize}\selectfont}%
  \ifx\svgwidth\undefined%
    \setlength{\unitlength}{150.14870828bp}%
    \ifx\svgscale\undefined%
      \relax%
    \else%
      \setlength{\unitlength}{\unitlength * \real{\svgscale}}%
    \fi%
  \else%
    \setlength{\unitlength}{\svgwidth}%
  \fi%
  \global\let\svgwidth\undefined%
  \global\let\svgscale\undefined%
  \makeatother%
  \begin{picture}(1,0.3879134)%
    \lineheight{1}%
    \setlength\tabcolsep{0pt}%
    \put(0,0){\includegraphics[width=\unitlength,page=1]{blowdown_handle.pdf}}%
    \put(0.44851228,0.08500975){\color[rgb]{0,0,0}\makebox(0,0)[lt]{\lineheight{1.25}\smash{\begin{tabular}[t]{l}$V_+\times\SSS^1$\end{tabular}}}}%
    \put(0.34943964,0.17651653){\color[rgb]{0,0,0}\makebox(0,0)[lt]{\lineheight{1.25}\smash{\begin{tabular}[t]{l}$V_-\times\SSS^1$\end{tabular}}}}%
    \put(0.04265291,0.19944023){\color[rgb]{0,0,0}\makebox(0,0)[lt]{\lineheight{1.25}\smash{\begin{tabular}[t]{l}$V_-\times D^2$\end{tabular}}}}%
    \put(0,0){\includegraphics[width=\unitlength,page=2]{blowdown_handle.pdf}}%
    \put(0.8713941,0.23682351){\color[rgb]{0,0,0}\makebox(0,0)[lt]{\lineheight{1.25}\smash{\begin{tabular}[t]{l}surgery  handles\end{tabular}}}}%
    \put(0,0){\includegraphics[width=\unitlength,page=3]{blowdown_handle.pdf}}%
    \put(0.75454808,0.35063926){\color[rgb]{1,0,0}\makebox(0,0)[lt]{\lineheight{1.25}\smash{\begin{tabular}[t]{l}$M_+'$\end{tabular}}}}%
    \put(0,0){\includegraphics[width=\unitlength,page=4]{blowdown_handle.pdf}}%
    \put(0.38027633,0.26224475){\color[rgb]{0,0,0}\makebox(0,0)[lt]{\lineheight{1.25}\smash{\begin{tabular}[t]{l}$N\times D^*\SSS^1$\end{tabular}}}}%
    \put(0,0){\includegraphics[width=\unitlength,page=5]{blowdown_handle.pdf}}%
  \end{picture}%
\endgroup%

%% file: s3.tex
\section{An algebraic perspective on tightness}
\label{sec:alg_tight}
Bourgeois and van-Koert \cite{MR2646902} showed that contact homology, which has now been defined rigorously on the homology level by  Pardon \cite{pardon2019contact} and Bao-Honda \cite{MR4539062}, vanishes for overtwisted contact manifolds. Motivated by this, Bourgeois and Niederkr\"uger \cite{bourgeois2010towards} introduced the notion of algebraically overtwisted contact manifolds as follows.

\begin{definition}\label{def:CHA}
    Let $Y$ be a contact manifold and $\CHA(Y;\Lambda)$ the contact homology of $Y$  over the Novikov field $\Lambda$. One says that:
    \begin{enumerate}
        \item $Y$ is \emph{algebraically overtwisted} if $\CHA(Y;\Lambda)=0$ \cite{bourgeois2010towards};
        \item $Y$ is \emph{algebraically tight} if $\CHA(Y;\Lambda)\ne 0$.
    \end{enumerate}
\end{definition}

\begin{proposition}\label{prop:property}
    We have the following properties.
\begin{enumerate}
    \item If $Y$ is overtwisted, then $Y$ is algebraically overtwisted. Hence algebraically tight contact manifolds are tight.
    \item  $Y$ is algebraically overtwisted if and only if one connected component of $Y$ is algebraically overtwisted. 
    \item Let $W$ be a Liouville cobordism with convex boundary $\partial_+W$ and concave boundary $\partial_-W$. If $\partial_+W$ is algebraically overtwisted, then so is $\partial_-W$.
    \item Strongly fillable contact manifolds are algebraically tight.
    \item The following contact manifolds are algebraically tight:
    \begin{enumerate}
        \item hypertight contact manifolds;
        \item contact manifolds with vanishing rational first Chern class and a contact form $\alpha$ such that there is no contractible Reeb orbit of SFT degree $1$;
        \item $1$-ADC contact manifolds.
    \end{enumerate}
\end{enumerate}
\end{proposition}
\begin{proof}
    The first claim follows from \cite{MR2646902} and the main result of \cite{CMP}. 
    Although \cite{MR2646902} uses $\Q$-coefficients, the claim is also true with $\Lambda$-coefficients (i.e.\ vanishing of contact homology is independent of these choices), as the power of the Novikov variable in the Novikov field is determined by the periods of Reeb orbits (see e.g.\  \cite[Proof of Corollary 3.15]{RSFT}).
    The second claim follows from the fact that $\CHA(Y_1\sqcup Y_2)=\CHA(Y_1)\otimes \CHA(Y_2)$. 
    The third claim follows from the functoriality of contact homology, i.e.\ we have an unital algebra map $\CHA(\partial_+W) \to \CHA(\partial_-W)$.
    The fourth claim follows from the fact that $\CHA(Y;\Lambda)=0$ is an obstruction to the existence of augmentations, hence the existence of strong fillings, see e.g.\ \cite[Corollary 3.15]{RSFT}. 
    Conditions in (a) and (b) of the fifth claim imply that $1\ne 0 \in \CHA(Y)$ by homotopy class or degree reasons, hence $\CHA(Y)\ne 0$.
    Roughly speaking, (c) is a special case of (b).
    (Recall that $k-$ADC contact structures have torsion $c_1$ by definition.)
    However, to cope with the asymptotic property in the definition of $1$-ADC, we can argue as follows. 
    If $\CHA(Y)=0$, then for any fixed contact form $\alpha$, there exists a positive number $A$, such that $1=0\in \CHA^{<A}(Y,\alpha)$, i.e.\ the homology of the sub-complex generated by $\alpha$-orbits with contact action at most $A$.
    By the $1$-ADC condition, we can find a contact form $\alpha_0<\alpha$, such that all $\alpha_0$-Reeb orbits with action smaller than $A$ have SFT degree at least $2$. 
    Then the functoriality $\CHA^{<A}(Y,\alpha)\to \CHA^{<A}(Y,\alpha_0)$ implies that $1=0\in  \CHA^{<A}(Y,\alpha_0)$, which is impossible due to degree reasons.
\end{proof}

\begin{example}[$1$-ADC Bourgeois manifolds]\label{ex:Bourgeois}
Starting with an open book whose binding is $(-1)$-ADC and simply connected, according to \Cref{lem:kADC_stable}, the associated Bourgeois contact structure is $1$-ADC, hence algebraically tight. As special cases, Bourgeois manifolds in \Cref{ex:Brieskorn_ADC} are algebraically tight. Then, after contact surgeries, the contact manifold remains algebraically tight, hence tight. 
\end{example}

\begin{remark}[Algebraic vs.\ geometric]
One of the most fundamental questions in symplectic and contact topology is whether the boundary between rigidity and flexibility is captured by pseudo-holomorphic curves. 
In this framework, it was shown by Avdek \cite{avdek2020combinatorial} that algebraic overtwistedness in dimension $3$ does not imply overtwistedness.
Avdek used another tightness criterion from Heegaard-Floer homology, which is still based on pseudo-holomorphic curves. 
Unlike the situation in dimension $3$, in higher dimensions, it seems that algebraic tightness is the only currently available criterion to ensure tightness.  
\end{remark}

In dimension $5$, we can reinterpret the proof of \cite[Theorem A]{BGM} to get the following.
\begin{theorem}\label{thm:alg_tight_Bour}
    $5$-dimensional Bourgeois contact structures $BO(\Sigma,\psi)$ are algebraically tight if $\Sigma$ is \emph{non-sporadic} (i.e.\ not a sphere with $3$ or less punctures).
\end{theorem}
\begin{proof}
    The hypothesis of non-sporadic page gives that there is a strong symplectic cobordism whose negative end consists of hypertight contact manifolds \cite[Theorem 3.1 and Corollary 4.1]{BGM}.
    Now the symplectic form on this cobordism is exact (hence no sphere bubbling can occur), however the given global primitive $1$-form is \emph{not} compatible with the contact structure on the negative end (this cobordism is called \emph{pseudo}-Liouville in \cite{BGM}). 
    In terms of functoriality in SFT, a priori this leads to deformations of the algebraic structure on the negative boundary by counting holomorphic caps in the cobordism, see \cite{MZ}.  However, holomorphic caps for this specific cobordism can be ruled out as shown in the proof of \cite[Lemma 5.1]{BGM}. Hence the usual functoriality for contact homology holds (over Novikov coefficients), and the positive boundary is also algebraically tight, as hypertight contact manifolds are algebraically tight.
\end{proof}
By point $(3)$ of \Cref{prop:property}, we then have the following useful consequence:
\begin{corollary}\label{cor:alg_tight_Bour}
   Any contact manifold obtained via contact surgery on a $5$-dimensional Bourgeois contact structure $BO(\Sigma,\psi)$ for non-sporadic $\Sigma$ is algebraically tight, and hence in particular tight.
\end{corollary}

%% file: s4_detail.tex
\section{Homological restrictions on symplectic fillings}

We begin by giving some context to our main technical result that puts restrictions on the homology of strong symplectic fillings of certain contact manifolds. For the moment we will avoid technicalities and refer to the following sections for more details.

\subsection*{Uniqueness results} Eliashberg--Floer--McDuff theorem \cite{McD91} showed that any Liouville filling of the standard sphere of dimension at least $5$ is diffeomorphic to a ball. In order to achieve this, one considers certain moduli spaces of holomorphic curves that serve as tools to determine the topology of the filling. In particular, one shows that the {\em homology} of the filling is fixed by that on the boundary. Uniqueness results of this kind have since been obtained for more general classes of manifolds and apply to fillings of the contact boundary of a split manifold of the form $V\times \D^2$, or to flexible Weinstein domains; cf.\ \cite{OV,BGZ,Zho21,product,BGM,GKZ}. 

If one considers the more general class of strong fillings, these are no longer unique, as one can take a blow-up of any given filling without affecting the boundary. In particular, one expects a lower bound on the {\em homological complexity} of a strong filling coming from a standard Liouville filling.

\subsection*{Symplectic cohomology and fillings}  In contrast to arguments based on the Eliashberg-Floer-McDuff result, there is another approach that uses symplectic cohomology. It is well known that the symplectic cohomology of any subcritical Weinstein domain vanishes \cite{Cie02b}, and this was extended to flexible Weinstein domains with the advent of Murphy's notion of loose Legendrians \cite{BourEkEl,zbMATH06864340}. In particular, for a flexible filling $W_0$, the tautological sequence yields
 $$\ldots  \to  SH^{*}(W_0)=0 \to SH_+^{*}(W_0) \stackrel{\cong}{\to} H^*(W_0)[1] \to SH^{*+1}(W_0)=0 \to \ldots$$
so that the positive symplectic cohomology is isomorphic to the ordinary cohomology of the filling with a grading shift.

\subsection*{Dynamics gives contact invariants and homological constraints} {\em A priori}, positive symplectic cohomology might depend on the choice of the Liouville filling. Bourgeois and Oancea \cite{MR2471597} showed that there is a way to understand the role of fillings in positive symplectic cohomology in terms of augmentations/counting rigid holomorphic planes from the filling. As a result, Cieliebak and Oancea \cite{CieOan18} utilised a neck-stretching argument to show that positive symplectic cohomology does not depend on the choice of filling $W$ (with mild topological assumptions) given the contact boundary is index positive as in \Cref{def:index-pos}. The same result was generalized to the boundary of flexible Weinstein domains with vanishing first Chern class by Lazarev \cite{Laz20}, as the contact boundary is ADC (cf.\ \Cref{def:kADC}). More precisely, under the index-positive/ADC assumption, the Floer differential in positive symplectic cohomology for suitable choices of contact forms and almost complex structures ``lives'' entirely in the end of the filling. In this way, one gets that the homology of a filling of a flexibly fillable contact manifold (of vanishing first Chern class) is an invariant.

In fact, much more is true: one can consider the restriction of the connecting homomorphism to the boundary, which also does not depend on the filling for ADC manifolds. 
$$ SH_+^{*}(W) \stackrel{\delta}{\to} H^*(W)[1]  \to H^*(\partial W)[1].$$
This then gives an important perspective on this invariant, namely we can understand the image of $H^*(W)\to H^*(\partial W)$ of any Liouville filling via the invariant map $\delta_{\partial}:SH_*^+(W)\to H^*(\partial W)[1]$ in a standard filling. This viewpoint was used to deduce topological information of Liouville fillings in \cite{Zho21}. Moreover, in the split case $W=V\times \DD^2$, similar results hold without any assumptions on the first Chern class \cite{product}.
\subsection*{General strong fillings} 
In order to extend these ideas to strong fillings and to provide filling obstructions to obtain \Cref{thm:tightnonfill_general}, one needs to extend the above discussion to this more general class of fillings and localize the properties under connected sums. The former concern requires a virtual setup: here we use Pardon's construction of Hamiltonian-Floer cohomology for a general symplectic manifold via his VFC theory \cite{pardon}, as symplectic cohomology is just a special case of Hamiltonian-Floer cohomology. The latter concern requires understanding the behavior of contact dynamics and holomorphic curves under connected sums. On the other hand, to infer information on the topology of the filling, it is sufficient to use the tautological exact sequence and the restriction on the connecting homomorphism for filtered symplectic cohomology (which involves Reeb orbits up to an action threshold), despite the fact that this is not an invariant of the contact manifold (as it depends on the contact form):
$$\ldots \to H^*(W) \to  SH^{*,<D}(W) \to SH_+^{*,<D}(W) \stackrel{\delta}{\to} H^*(W)[1] \to \ldots$$
In particular,  we make essential use of specific models for the Reeb dynamics, and one can think of this as arguing with cohomology on the chain level rather than with the invariant directly. In particular, our arguments involve careful computations of various boundary maps {\em on the nose}.

\subsection{Main result}\label{sec:obstr_fillings}
The main tool that we will use to obstruct strong symplectic fillings is given by the following:
\begin{theorem}[Homological Constraints on Fillings]\label{thm:unique}
\ 
Let $W_0$ be a $(2n+2)$-dimensional flexible Weinstein domain with $n\geq 2,$ and let $(M,\xi)$ be any contact manifold.
    \begin{enumerate}
        \item[(a)]  Suppose that $c_1(W_0)$ and $c_1(\xi)$ both vanish. If $W$ is a strong filling of $Y=M\# \partial W_0$, then the kernel of $H_n(\partial W_0;\Q)\to H_n(W_0;\Q)$ contains the kernel of $$H_n(\partial W_0;\Q) \stackrel{0\oplus \Id }{\hookrightarrow}  H_n(M;\Q) \oplus H_n(\partial W_0;\Q)=H_n(M\# \partial W_0;\Q)\rightarrow H_n(W;\Q).$$

   \item[(b)] Let $W_0$ be subcritical (\textbf{without} any assumption on $c_1=0$).
   Then, for any strong filling $W$ of $Y$ we have that $H_n(\partial W_0;\Q)\to H_n(W;\Q)$ is injective.
    \end{enumerate}
\end{theorem}

Note that, in the case that $W_0$ is subcritical, the first claim reduces to the fact that $H_n(\partial W_0;\Q)\to H_n(W;\Q)$ is injective.
Indeed, the first map $H_n(\partial W_0;\Q)\to H_n(W_0;\Q)$ is automatically injective in this case:
this is a consequence of the long exact sequence of the homology of the pair, as $H_{n+1}(W_0,\partial W_0 ; \Q)$ is generated by the cocores of the Weinstein handles of index $n+1$, and hence there are none if $W_0$ is subcritical.
The non-trivial content of statement (b) is that subcriticality implies that $H_n(\partial W_0;\Q)\to H_n(W;\Q)$ is injective \emph{without the assumption that $c_1(\xi)$ and $c_1(W_0)$ vanish}.

\begin{remark}
	\Cref{thm:unique} is stated in this way only for the purpose of constructing exotic contact structures on spheres in this paper, but in fact, a stronger version of this result, giving homological information on the fillings in some other degrees for $W_0$ flexible, can also be obtained by similar arguments. 
    We also point out that, when $M=\emptyset$ and $W_0$ is subcritical, one can use a strategy of filling by pseudo-holomorphic curves similar to \cite{BGZ} to obtain a similar result, which was used by the first three authors \cite[Theorem B]{BGM} to obstruct strong fillings.
\end{remark}
In the proof of \Cref{thm:unique}, it is important to rule out the effects from Reeb orbits on $M$ as well as the boundary of the Weinstein $1$-handle in the connected sum. The former can be easily dealt with if we scale up the contact form on $M$, and those Reeb orbits will have large periods and hence not be seen, as we will use a fixed period threshold for filtered symplectic cohomology on $\partial W_0$. It is the Reeb orbits on the handle that divides \Cref{thm:unique} into two cases. In the general flexible case, the Reeb orbits created on the handle will have sufficiently high Conley-Zehnder indices, see e.g.\ \Cref{lemma:connected_sum}. Those orbits do not affect the piece of symplectic cohomology we care about. This uses the $\Z$-grading in an essential way. In the subcritical case, we make use of the splitting form and a special choice of almost complex structure so that we can rule out the interaction with those orbits on the handle by an energy argument, which is inspired by similar considerations in the work of Chiang-Ding-van Koert \cite{CDvK16}.

\subsection{Symplectic Cohomology}
\label{sec:basics_sympl_cohomol}
As outlined above, the main tool to prove \Cref{thm:unique} will be symplectic cohomology. Moreover, we will need a precise description of certain moduli spaces. We then start by recalling the basics of symplectic cohomology (we refer readers to \cite{CieOan18} for a more detailed account). 
\subsection*{Basic Properties}
Let $W$ be a Liouville filling with \emph{strict} contact boundary $(Y,\alpha)$ (i.e.\ we fix the contact form $\alpha$ such that the Liouville form of $W$ is restricted to be $\alpha$ on the boundary). 
\begin{enumerate}
    \item (Filtered Symplectic Cohomology) For any $D>0$, which is not the period of some Reeb orbit of $\alpha$, we have a $\Z/2-$graded filtered positive symplectic cohomology $SH_+^{*,<D}(W)$. We may assume Reeb orbits of period at most $D$ are non-degenerate. 
    Its underlying cochain complex is generated by the hat $\hat{\gamma}$ and check $\check{\gamma}$ versions of the Reeb orbits with period up to $D$, whose grading is given by $\frac{\dim W}{2}-\mu_{CZ} \mod 2$; here, $\mu_{CZ}(\check{\gamma})=\mu_{CZ}(\hat{\gamma})-1=\mu_{CZ}(\gamma)$ with $\mu_{CZ}(\gamma)$ being the Conley-Zehnder index of the Reeb orbit $\gamma$. The direct limit w.r.t.\ $D\to +\infty$ yields the symplectic cohomology of $W$, which is an invariant of $W$ up to Liouville homotopy. On the other hand, the filtered positive symplectic cohomology $SH_+^{*,<D}(W)$ depends on the boundary contact form $\alpha$ and is not an invariant of $W$ up to homotopy. Note that we suppress the appearance of $\alpha$ into the assumption that $W$ is a Liouville filling with \emph{strict} contact boundary $(Y,\alpha)$ for simplicity. We also have the filtered symplectic cohomology $SH^{*,<D}(W)$, whose cochain complex is generated by the above orbits, and a Morse cochain complex of $W$. Moreover, these cohomologies fit into a tautological long exact sequence:
    $$\ldots \to H^*(W) \to  SH^{*,<D}(W) \to SH_+^{*,<D}(W) \stackrel{\delta}{\to} H^*(W)[1] \to \ldots.$$
    \item (Boundary Map) We define $\delta_{\partial}: SH_+^{*,<D}(W) \to H^*(Y)[1]$ to be the composition of the connecting map $\delta$ and the restriction map in cohomology associated to the inclusion $Y\hookrightarrow W$. 
    Then, according to \cite[\S 3.1]{Zho21}, $\delta_{\partial}$ can be defined directly by counting rigid configurations consisting of a holomorphic plane in $W$ and a tail of gradient flow line in $Y$, with respect to any auxiliary Morse function on $Y$. 
    When $\delta_{\partial}(x)$ has a non-trivial component in $H^d(Y)$, then we have
    $$d-\frac{1}{2}\dim W + \mu_{CZ}(x)-1=0,$$
    where $\mu_{CZ}(x)$ is the Conley-Zehnder index computed with respect to the trivialization induced by the holomorphic plane which contributes to said non-trivial component in $H^d(Y)$.
    
    \item\label{item:VFC}(Novikov Coefficients) Symplectic cohomology for Liouville fillings is classically defined for Liouville fillings over $\Z$, and more generally for semi-positive symplectic fillings with coefficients in the Novikov ring over $\Z$. 
    Appealing to virtual perturbation techniques (e.g.\ polyfolds, virtual fundamental cycles, Kuranishi techniques) one can furthermore define symplectic cohomology in the case of general symplectic fillings; the drawback is that it is now defined over the Novikov \emph{field} $\Lambda$ over $\Q$, where 
    $$\Lambda=\left\{\sum_{i=1}^{\infty} a_iq^{\lambda_i}\left| a_i\in \Q,\lambda_i\to +\infty \right. \right\}.$$
    This can be done also for positive symplectic cohomology, leveraging the asymptotic behaviour lemma \cite[Lemma 2.3]{CieOan18}. 

    For the purposes of this paper, it is enough to use the construction of Hamiltonian-Floer cohomology for general symplectic manifolds in \cite{pardon}: more precisely, 
    the results in this paper rely only on the fact that a virtual count/perturbed count can be arranged to be the geometric count when transversality holds, e.g.\ by \cite[Proposition 4.33]{pardon2019contact} in the case of Pardon's VFC. 
    One then also has an analogous tautological exact sequence (with Novikov coefficients):
      $$\ldots \to H^*(W,\Lambda) \to  SH^{*,<D}(W,\Lambda) \to SH_+^{*,<D}(W,\Lambda) \stackrel{\delta}{\to} H^*(W,\Lambda)[1] \to \ldots.$$
      The Novikov coefficient cohomology of $W$ is simply the tensor product $H^*(W)\otimes_{\Q}\Lambda$.
\end{enumerate}

\subsection{Neck-stretching}
We will repeatedly apply neck-stretching \cite{BEHWZ03} to moduli spaces in symplectic cohomology. 
    Roughly speaking, we will use this to reduce the computation of boundary/connecting maps to counting certain curves in the symplectization (with an appropriate choice of Hamiltonian) and hence in a standard model which is ``sufficiently stretched". Here we present a very brief account of how to apply neck-stretching to moduli spaces in symplectic cohomology, and refer readers to \cite[Lemma 2.4, Proposition 9.17 ]{CieOan18} for further details (see also \cite[\S 3.2]{Zho21} or \cite[\S 2.5]{product}).  

\medskip
    
    Let $(Y,\alpha)$ be a contact-type hypersurface in a symplectic filling $W$. 
    Assume that $Y$ divides $W$ into the union of a cobordism $X$ from $Y$ to $\partial W$ and a filling $W'$ of $Y$. 
    
    For any almost complex structure $J$, which is compatible with the contact structure near $Y$, we can find a $[0,1)-$family of compatible almost complex structures on $\widehat{W}$ starting at $J_0=J$ and converging, for $t\to1$, to the almost complex structures on the completions $\widehat{X}$, $\widehat{W'}$ and $Y\times \RR_+$ obtained by stretching $J$ in the Liouville vector direction along $Y$.  

\begin{figure}[t]
	\begin{center}
		\begin{tikzpicture}[scale=0.5]
		\path [fill=blue!15] (0,0) to [out=20, in=160]  (6,0) to [out=270,in=90] (6,-6) to [out=170,in=10] (0,-6) to [out=90, in=270] (0,0);
		\path [fill=red!15] (0,-6) to [out=10, in=170]  (6,-6) to [out=270,in=0] (3,-10) to [out=180,in=270] (0,-6); 
		\draw (0,0) to [out=20, in=160]  (6,0) to [out=270,in=90] (6,-6) to [out=170,in=10] (0,-6) to [out=90, in=270] (0,0);
		\draw (6,-6) to [out=270,in=0] (3,-10) to [out=180,in=270] (0,-6); 
		\draw[dashed] (0, -5.5) to [out=10, in=170] (6,-5.5);
		\draw[dashed] (0.02,-6.5) to [out=10, in=170] (5.98,-6.5);
		\draw (2,-1) to [out=90, in=180] (2.5, -0.75) to [out=0, in = 90] (3,-1) to [out=270,in=0] (2.5,-1.25) to [out=180,in=270] (2,-1);
		\draw (2,-1) to [out=270,in=90] (1,-7) to [out=270,in=180] (2,-8) to [out=0,in=180](2.5,-3) to [out=0, in=90](3,-4);
		\draw (3,-1) to [out=270, in=90] (4,-4) to [out=270, in=0] (3.5,-4.25) to [out=180,in=270] (3,-4);
		\draw[dotted] (3,-4) to [out=90, in=180] (3.5,-3.75) to [out=0, in=90] (4,-4);
		
		\node at (2.5,-2) {$u$};
		\end{tikzpicture}
		\hspace{1cm}
		\begin{tikzpicture}[xscale=0.5,yscale=0.7]
		\path [fill=blue!15] (0,-1) to [out=20, in=160]  (6,-1) to [out=270,in=90] (6,-6) to [out=170,in=10] (0,-6) to [out=90, in=270] (0,-1);
		\path [fill=red!15] (0,-6) to [out=10, in=170]  (6,-6) to [out=270,in=0] (3,-10) to [out=180,in=270] (0,-6); 
		\draw (0,-1) to [out=20, in=160]  (6,-1) to [out=270,in=90] (6,-6) to [out=170,in=10] (0,-6) to [out=90, in=270] (0,-1);
		\draw (6,-6) to [out=270,in=0] (3,-10) to [out=180,in=270] (0,-6);
		\draw[dashed] (0, -5.5) to [out=10, in=170] (6,-5.5);
		\draw[dashed] (0.1, -6.5) to [out=10, in=170] (5.9,-6.5);
		\draw (2,-2) to [out=90, in=180] (2.5, -1.8) to [out=0, in = 90] (3,-2) to [out=270,in=0] (2.5,-2.2) to [out=180,in=270] (2,-2);
		\draw (2,-2) to [out=270,in=90] (1,-7) to [out=270,in=180] (2,-8) to [out=0,in=180](2.5,-3.5) to [out=0, in=90](3,-4);
		\draw (3,-2) to [out=270, in=90] (4,-4) to [out=270, in=0] (3.5,-4.2) to [out=180,in=270] (3,-4);
		\draw[dotted] (3,-4) to [out=90, in=180] (3.5,-3.8) to [out=0, in=90] (4,-4);

		\draw (5.75,-5) to (5.8,-5);
		\node at (2.5,-3) {$u$};
		\end{tikzpicture}
		\hspace{1cm}
		\begin{tikzpicture}[scale=0.5]
		\path [fill=blue!15] (0.5,-6) to [out=90, in=270]  (0,0) to [out=20,in=160] (6,0) to [out=270,in=90] (5.5,-6) to [out=160, in=20] (0.5,-6);
		\path [fill=purple!15] (0.5,-6.2) to [out=20,in=160] (5.5,-6.2) to [out=270,in=90] (5.5,-12) to [out=160, in=20] (0.5,-12) to [out=90, in=270] (0.5,-6.2);
		\path [fill=red!15] (0.5, -12.2) to [out=20,in=160] (5.5,-12.2) to [out=270,in=0] (3,-16) to [out=180,in=270] (0.5,-12.2);
		\draw (0.5,-6) to [out=90, in=270]  (0,0) to [out=20,in=160] (6,0) to [out=270,in=90] (5.5,-6);
		\draw [dashed] (5.5,-6) to [out=160, in=20] (0.5,-6);
		\draw[dashed] (0.5,-6.2) to [out=20,in=160] (5.5,-6.2);
		\draw (5.5,-6.2) to [out=270,in=90] (5.5,-12);
		\draw[dashed] (5.5,-12) to [out=160, in=20] (0.5,-12);
		\draw (0.5,-12) to [out=90, in=270] (0.5,-6.2);
		\draw [dashed](0.5, -12.2) to [out=20,in=160] (5.5,-12.2); 
		\draw (5.5,-12.2)to [out=270,in=0] (3,-16) to [out=180,in=270] (0.5,-12.2);
		\draw (2,-1) to [out=90, in=180] (2.5, -0.75) to [out=0, in = 90] (3,-1) to [out=270,in=0] (2.5,-1.25) to [out=180,in=270] (2,-1);
		\draw (2,-1) to [out=270,in=90] (1,-4) to [out=270, in=180]  (1.5,-5.8) to [out=0,in=270](2,-5) to [out=90,in=180] (2.5,-3) to [out=0, in=90](3,-4);
		\draw (3,-1) to [out=270, in=90] (4,-4) to [out=270, in=0] (3.5,-4.25) to [out=180,in=270] (3,-4);
		\draw[dotted] (3,-4) to [out=90, in=180] (3.5,-3.75) to [out=0, in=90] (4,-4);

		\draw (1.5,-5.9) to [out=180,in=90] (0.8,-9) to [out=270, in=180] (1.5,-11.8) to [out=0,in=270] (2.2,-9) to [out=90, in=0] (1.5,-5.9);
		\draw (1.5,-11.9) to [out=180,in=90] (0.8,-13) to [out=270, in=180] (1.5,-14) to [out=0,in=270] (2.2,-13) to [out=90, in=0] (1.5,-11.9);

		\node at (1.5,-5.8) [circle, fill=white, draw, outer sep=0pt, inner sep=3 pt] {};

		\node at (1.5,-11.8) [circle, fill=white, draw, outer sep=0pt, inner sep=3 pt] {};

		\node at (2.5,-2) {$u^{\infty}$};
		\node at (4.5,-2) {$\widehat{X}$};
		\node at (4,-10) {$\widehat{Y}$};
		\node at (4,-14) {$\widehat{W'}$};
		\end{tikzpicture}
	\end{center}
    \caption{Neck-stretching. We use $\bigcirc$ to indicate the puncture that is asymptotic to a Reeb orbit. The cylindrical ends are asymptotic to Reeb orbits, and $u^{\infty}$ is the top-level curve in the fully stretched picture.}
    \label{fig:neck}
\end{figure}
    
    Holomorphic curves and/or Floer cylinders for this family of almost complex structures $J_t$, will then converge to SFT buildings for $t\to1$, in such a way that the top-level curve, i.e.\ the curve contained in $\widehat{X}$ will develop negative punctures asymptotic to Reeb orbits on $Y$, see \Cref{fig:neck}.  
    Denoting by $\Gamma$ the set of Reeb orbits to which these negative punctures are asymptotic, 
    we then have the following two constraints on the top-level curve:
   
    \begin{enumerate}
        \item\label{item:action_constraint} \textbf{Action constraint}: {\it action increases along the top-level curve, i.e.\
        \begin{equation}
        \label{eqn:action_constraint}
        \mathcal{A}_H(\text{output end})-\mathcal{A}_H(\text{input end})-\sum_{\gamma\in \Gamma }\int \gamma^*\alpha \; \ge 0 \, ,
        \end{equation}
        and the equality holds if and only if:
        \begin{itemize} 
        \item the top-level curve $u_\infty$ on $u^{-1}_\infty(Y\times (0,1))$ is contained in $\gamma \times (0,1)$, where $Y\times (0,1)_r$ is the negative end of the symplectization  $Y\times(-\infty,0)_s$ with $r=e^{-s}$ and $\gamma$ is a Reeb orbit on $Y$, 
        \item and $u_\infty$ is independent of $s$ on $u_{\infty}^{-1}(\widehat{X}\backslash Y\times(0,1))\subset \RR_s\times \SSS^1_t$.
        \end{itemize}
        }

        \medskip
        
        In the above inequality, $\mathcal{A}_H$ is the symplectic action, which is well-defined if $\widehat{X}$ is exact, for the (time-dependent) Hamiltonian $H$. Here $H$ is zero on $X\cup Y\times (0,1)$ and with a certain slope near the positive cylindrical end of $\widehat{X}$ as in the setup of symplectic cohomology used in \cite{zhou2023contact}.
        Namely, if $\lambda$ denotes the Liouville form, for any contractible loop $x\colon \SSS^1\to X$ one has
        \[
        \mathcal{A}_H(x) = -\int_{S^1}x^*\lambda + \int_{S^1}H_t\circ x(t) \mathrm{d} t \, .
        \]

        As an explicit proof of this inequality is hard to find in the literature, we give the argument below for the reader's convenience.

        \medskip

        \item \textbf{Index constraint}: {\it the top-level curve must have a non-negative expected dimension.
        More precisely, when $c_1(X)$ is torsion, then 
        \begin{equation}
        \label{eqn:index_constraint}
        m - \sum_{\gamma\in \Gamma}\left(\mu_{CZ}(\gamma)+\frac{\dim W}{2}-3\right)\ge 0 \, ,
        \end{equation}
        where $m$ is the expected dimension of the moduli space without negative punctures (which can be expressed using the Conley-Zehnder indices of the asymptotic Hamiltonian orbits depending on the moduli space under consideration), and $\mu_{CZ}(\gamma)+\frac{\dim W}{2}-3$ is the SFT degree of $\gamma$, see e.g.\ \cite[Proof of Proposition 9.17]{CieOan18}.}

        \medskip
        
        Note that when we speak of the expected dimension, we will always consider the unparametrized moduli space where we quotient out any translation symmetries.
        
        The constraint in \eqref{eqn:index_constraint} follows from a transversality argument, as the top-level curve in all the situations we will consider involves Hamiltonian perturbations and transversality can be achieved for a generic almost complex structure, as those Floer curves have somewhere injective points.
        
    \end{enumerate}

    \medskip

\begin{proof}[Proof of the action constraint]
        In this proof, we write $u$ for the top-level curve $u_\infty$ for simplicity.
        First, notice that there is an inequality 
        \begin{equation}
        \label{eqn:integral_tot_action_constraint_proof}
        \int_{u^{-1}(\widehat{X}\backslash Y\times (0,1))} |\partial_s u|^2+
        \int_{u^{-1}(Y\times (0,1))} u^*\rd \alpha    \ge 0 \, , 
        \end{equation}
        where $Y\times (0,1)_r$ is the negative end of the completion $\widehat{X}$ with Liouville form $r\alpha$. 
        The non-negativity of the first integral is obvious and the one of the second follows from the fact that we take $H$ to be zero on $Y\times (0,1)$, as in the setup of symplectic cohomology in \cite{Zho21}, so that there $u$ solves the Cauchy-Riemann equation with respect to $J$, that is tamed by $\rd\alpha$ on $\xi$. 

        Note that the $s=+\infty$ end of the punctured cylinder in the Floer equation is the input end.
        Then, because $u$ solves the Floer equation and $\lambda\vert_{Y\times \{1 \}}=\alpha$, using Stokes' theorem we can rewrite 
        \begin{eqnarray*}
            \int_{u^{-1}(\widehat{X}\backslash Y\times (0,1))} |\partial_s u|^2 \rd s \rd t & = &    \int_{u^{-1}(\widehat{X}\backslash Y\times (0,1))} \omega (\partial_su,J_t\partial_su)\rd s\rd t \\
            & = & \int_{u^{-1}(\widehat{X}\backslash Y\times (0,1))} \omega (\partial_su,\partial_tu-X_{H_t}(u)) \rd s\rd t\\
            & = &  \int_{u^{-1}(\widehat{X}\backslash Y\times (0,1))} u^*\rd \widehat{\lambda}-\rd H_t(\partial_su)\rd s\rd t\\
            & = &   \mathcal{A}_H(\text{output end})-\mathcal{A}_H(\text{input end})-\int_{u^{-1}(Y\times \{1\})}\alpha 
        \end{eqnarray*}
        The last equality follows from Stokes' theorem and the fact that $H_t$ vanishes on $u^{-1}(Y\times \{1\})$. Here, we implicitly assume $u^{-1}(Y\times \{1\})$ is cut out transversely; if not, by Sard's theorem, one can just consider $u^{-1}(Y\times \{1-\delta\})$ for a generic $0<\delta \ll 1$ instead.
        
        Similarly, the second integral is 
        \begin{equation}
        \label{eqn:integral_2_action_constraint_proof}
        \int_{u^{-1}(Y\times (0,1))} u^*\rd \alpha  = \int_{u^{-1}(Y\times \{1\})}\alpha-\sum_{\gamma\in \Gamma }\int \gamma^*\alpha \; .
        \end{equation}
        
        The claimed action constraint then follows directly by adding the above two expressions.
        The equality case also follows from the fact that the first integrand in \eqref{eqn:integral_tot_action_constraint_proof} is the squared norm of $\partial_s u$, and the non-negative form $u^*\rd \alpha$ is $0$ if and only if $\rd u$ is in $\la R_{\alpha},\partial_r\ra$, where $R_{\alpha}$ is the Reeb vector field of $\alpha$.
        \end{proof}

\bigskip
    
\subsection{Model for computing the connecting map $\delta_\partial$}\label{SS:4.4}
We give an explicit description of the moduli spaces that arise when we apply neck-stretching along the contact boundary of a symplectic manifold to the curves used to define the connecting map $\delta_\partial$. 

\medskip

Let $Y$ be a contact manifold equipped with a non-degenerate contact form $\alpha$. 
In addition let $H$ be a (time-dependent) Hamiltonian on $\widehat{Y}:=\RR_t\times Y \simeq (\RR_+)_r\times Y$, 
where $r=e^t$, such that $H$ is $0$ on $(0,1]_r\times Y$ and is the same as the standard Hamiltonian with slope $D$ on the cylindrical end $[2,+\infty)\times Y$ as in the usual definition of symplectic cohomology.
The almost complex structure $J$ is also assumed to be cylindrical in the cylindrical end, as in the usual symplectic cohomology setup.
We will refer to such a pair $(H,J)$ as \textbf{admissible}. 

Given a multiset $\Gamma=\{\gamma_1,\ldots,\gamma_k\}$ of Reeb orbits and \textbf{non-constant} Hamiltonian orbits $x,y$, we consider the moduli spaces
\begin{equation}\label{eqn:MYH}
    \cM_{Y,H}(x,y,\Gamma) = \bigslant{\left\{
    u:\RR_s\times \SSS^1_t\backslash \{p_1,\ldots,p_k\} \to \widehat{Y} \;
    \left|
    \begin{array}{c}
        \partial_s u+J(\partial_t u-X_{H})=0, \; \lim_{s\to \infty} u(s,\cdot)=x, \\
        \lim_{s\to -\infty} u(s,\cdot)=y, \; \lim_{p_i} u= \gamma_i 
    \end{array}\right.\right\}
    }{\RR} \, .
\end{equation}
Here, the notation $\lim_{p_i} u= \gamma_i$ means that $u$ is asymptotic to $\gamma_i$ at $p_i$ viewed as a negative puncture with a free asymptotic marker mapped to a chosen base point on $\mathrm{im} \gamma_i$; recall that in this region the Hamiltonian is zero and the equation is the usual Cauchy-Riemann equation. The $\RR$-action acts on the domain curve, which will move $p_i$ as well.  

In the cases of interest for us, these moduli spaces will either be smooth $0$-dimensional manifolds or will have negative virtual dimension and thus be empty for a generic choice of admissible $(H,J)$.
Moreover, in this case, the compactification $\overline{\cM}_{Y,H}(x,y,\Gamma)$ of $\cM_{Y,H}(x,y,\Gamma)$ is a mixture of Floer-type breaking at \emph{non-constant} Hamiltonian orbits into curves of the form \eqref{eqn:MYH} and of SFT building breaking at the lower level. Note that no breaking can occur at a constant orbit of $H$ due to the symplectic action constraints. 
We also point out that SFT-type breaking could appear as follows: if in a sequence of curves the $s$-coordinate of the puncture $p_i$ is going to $-\infty$ or $+\infty$, then after bringing back $p_i$ to a compact part of the cylinder using the $\RR$-translation we must see a Floer-type breaking; this particular moduli space of broken configurations is precisely what may appear as the top-level curve after applying neck-stretching to the Floer cylinders corresponding to the differential from $x$ to $y$ in the positive symplectic cohomology.

\medskip

Similarly, we can consider the neck-stretching limit of curves defining $\delta_{\partial}:SH^*_+(W)\to H^*(\partial W=Y)[1]$. We fix a Morse function $f$ on $Y$ along with a generic metric, for a fixed $\eta>0$ sufficiently small and a critical point $p$ of $f$, we will need to consider the compactification $\overline{\cM}_{Y,H}(x,S_p,\Gamma)$  of the following moduli space: 
\begin{equation}\label{eqn:MYP}
    \cM_{Y,H}(x,S_p,\Gamma) = \bigslant{\left\{u:\CC \backslash \{p_1,\ldots,p_k\} \to \widehat{Y}=\RR_+\times Y\left|
	\begin{array}{c}
	\partial_s u+J(\partial_t u-X_{H})=0, \;\,  p_i\ne 0,  \vspace{5pt} \\
	  \displaystyle \lim_{s\to \infty} u(s,\cdot)=x, \;\, \lim_{p_i} u= \gamma_i, \vspace{2pt} \\
   u(0)\in \{ 1-\eta\}\times S_p,
	\end{array}\right.\right\}}
 {\RR} \, ,
\end{equation}
where $S_p$ is the stable manifold of $p$, i.e.\ points that converge to $p$ under the gradient flow of $f$, and
the meaning of $\lim_{p_i}u=\gamma_i$ is the same as before. 
In other words, elements in $\cM_{Y,H}(x,S_p,\Gamma)$ consist of a curve $u$ from a punctured plane that is asymptotic to the Hamiltonian $x$ at the input end and to Reeb orbits at the extra negative punctures modulo the $\RR$-translation along with a tail of gradient line of $f$ from $u(0)$ to $p$ in $\{1-\eta\}\times Y$. 
The compactification  $\overline{\cM}_{Y,H}(x,S_p,\Gamma)$ involves Floer breaking at the input end into curves from \eqref{eqn:MYH} and \eqref{eqn:MYP}, SFT building breaking at the negative punctures as well as Morse flow line breaking at the tail.
Note that in this setting degenerations where one of the punctures goes to $0\in \CC$ can also happen, and they lead to SFT/Floer buildings as discussed in the case of the previous moduli space for a puncture escaping at infinity in the cylinder.
\medskip

\subsection{Reeb dynamics and symplectic cohomology of connected sums}
The Reeb dynamics that result from a connected sum can be described explicitly and were worked out by Yau \cite{Yau}, see also \cite{Laz20}. The following can be proved by combining \cite[\S 4]{Yau} and \cite[Proof of Proposition 4.4]{Laz20}. 
\begin{lemma}\label{lemma:connected_sum}
    Let $(Y_1,\alpha_1)$ and $(Y_2,\alpha_2)$ be two $(2n+1)$-dimensional contact manifolds with contact forms whose Reeb orbits do not cover the whole manifold $Y_1,Y_2$. Then, by picking two points $p_1\in Y_1,p_2\in Y_2$ that do not lie on Reeb orbits, we can attach a Weinstein $1$-handle to them to obtain a contact connected sum $(Y_1\#Y_2,\alpha_1\#_D\alpha_2)$ for $D\gg 0$ with the following properties:
    \begin{enumerate}
        \item All Reeb orbits of $\alpha_1 \#_D \alpha_2$ of period at most $D$ are either from $Y_1$ and $Y_2$ outside the handle attachment region or multiple covers $\gamma_{h,i}^k$ of the Reeb orbit $\gamma_{h,i}$ for $1\le i \le n$ contained in a standard contact sphere of dimension $2n-1$ in the co-core of the $1$-handle, for $k\le k_0=\lfloor D/ \int \gamma_{h,i}^*(\alpha_1\#_D\alpha_2) \rfloor$ (for any $i$, as $\gamma_{h,i}$ has approximately the same small period). 
        \item The Conley-Zehnder index of $\gamma_{h,i}^k$ is $n-1+2(k-1)n+2i \ge n+1$, where the Conley-Zehnder index is computed from bounding discs contained in the handle. That is, the Conley-Zehnder index of those orbits enumerates through $\{n+1,n+3, \ldots, 2k_0n+n-1\}$. It is the canonical global Conley-Zehnder index if $c_1(Y_1)=c_1(Y_2)=0$.
    \end{enumerate}
    The handle grows thinner as $D\to +\infty$.
\end{lemma}
The purpose of assuming that the base points $p_1,p_2$ are not on a closed Reeb orbit is to find sufficiently small Darboux neighborhoods such that the Reeb flow from the Darboux ball will not return within time $D$. Those orbits $\gamma_{h,i}^k$ are from a standard contact sphere of codimension $2$ contained in the co-core of the $1$-handle with an ellipsoid contact form that is close to a round contact form; hence there are precisely $n$ simple Reeb orbits with approximately the same period. Their Conley-Zehnder index is precisely the Conley-Zehnder index on the ellipsoid of dimension $2n-1$, as the linearized flow in the symplectic normal direction is positive hyperbolic, see \cite[Lemma 3.1, Theorem 3.1]{Yau} for the computation in the handle model. 

The effect of contact connected sum on symplectic cohomology is well-understood by the work of Cieliebak \cite{Cie02}. The following is just a filtered version of the corresponding result in symplectic cohomology and was implicitly proved in \cite{Cie02}; see \cite[Proposition 9.19]{CieOan18} for an equivalent statement formulated using symplectic cohomology of cobordisms. We only state a special case for the purpose of \Cref{thm:unique}.
\begin{proposition}\label{prop:symp_sum}
    Assume $W_1$ and $W_2$ are $(2n+2)$-dimensional Liouville fillings of $Y_1$ and $Y_2$ respectively for $n\ge 2$. Let $W_1\natural W_2$ be the strict Liouville filling of $(Y_1\#Y_2, \alpha_1\#_D\alpha_2)$ from \Cref{lemma:connected_sum} (for any $D$ such that \Cref{lemma:connected_sum} applies) such that $c_1(W_1)=c_1(W_2)=0$. Then we have 
    $$SH^{n,<D}(W_1\natural W_2)\simeq SH^{n, <D}(W_1)\oplus SH^{n,<D}(W_2)$$
    via the Viterbo transfer map.
\end{proposition}
\begin{proof}
    We give two explanations of the result, as we believe both are of interest for different reasons: the first uses the specific degree in the statement (i.e.\ the $SH^*$ grading $n$, or equivalently, Conley-Zehnder index $1$); the second uses the general strategy of \cite{Cie02}. 
    
    {\bf Proof 1:} The Conley-Zehnder index of $\gamma_{h,i}^k$ is at least $n+1\ge 3$. Hence, all these orbits have cohomological grading\footnote{I.e.\ the degree in $SH^*$, namely $n+1-\mu_{CZ}$.} at most $0$. Therefore $W_1\natural W_2$ and $W_1\sqcup W_2$ have the same generators with cohomological grading $n,n+1$. They also share the same generators with cohomological grading $n-1$ if $n\ge 3$ and when $n=2$, $W_1\natural W_2$ has an extra grading $1$-generator corresponding to the Morse critical point in the $1$-handle of $W_1\natural W_2$. Using the action filtration as in \cite[\S 4.2]{zhou2023contact}, the Viterbo transfer map from $W_1\natural W_2$ to $W_1\sqcup W_2$ induces an isomorphism on positive cochains (i.e.\ those with generators from Reeb orbits) of cohomological grading $n-1,n,n+1$. Therefore the Viterbo transfer map from $W_1\natural W_2$ to $W_1\sqcup W_2$ induces a cochain isomorphism of cohomological grading $n,n+1$ and surjection on $n-1$, which then gives an isomorphism $SH^{n,<D}(W_1\natural W_2)\simeq SH^{n, <D}(W_1)\oplus SH^{n,<D}(W_2)$.

    {\bf Proof 2:}  Alternatively, the filtered symplectic cohomology can be defined using any Hamiltonian whose slope is the period threshold on the cylindrical end, see \cite[Proposition 2.8]{Zhou_2021}. Thanks to \cite[Lemma 2.5]{Cie02b} or \cite[\S 3.3]{Fau20}, one can then find a Hamiltonian on $W_1\natural W_2$ defining $SH^{*,<D}(W_1\natural W_2)$, such that the generators are those of $W_1\sqcup W_2$ defining filtered symplectic cohomology of $W_1\sqcup W_2$ along with one extra generator -- a critical point -- in the $1$-handle with Conley-Zehnder index $\gg n$\footnote{The Conley-Zehnder index of such a generator is actually $n+2k_0n$, where $k_0$ is as in \Cref{lemma:connected_sum}. 
    This index can be understood as the index of the survivor in cohomology of the complex of check and hat orbits from (2) of \Cref{lemma:connected_sum}, which is the hat version of the largest index in (2) of \Cref{lemma:connected_sum}. The Hamiltonian is not $C^2$-small near the critical points. In the context of $0$ handles, this is \cite[(3f)]{biased}. The case of $1$-handles is precisely the product of a codimension $2$  $0$-handle and $\CC$ with a hyperbolic point. }. That is the unique critical point in the handle with a large positive Hessian in a codimension $2$ symplectic subspace and a hyperbolic Hessian in the normal direction, see \cite[Lemma 2.5 with suitable $A_i,B_i$]{Cie02b}. By \cite[Theorem 1.11, before taking the limit w.r.t. slopes]{Cie02b}, we get isomorphism 
    $$SH^{n,<D}(W_1\natural W_2)\simeq SH^{n, <D}(W_1)\oplus SH^{n,<D}(W_2).$$
    The second approach works for any degree, not just $n$, see \cite[Proposition 9.19]{CieOan18}, where the statement is for full symplectic cohomology while the proof factors through the filtered version.
\end{proof}

\subsection{Moduli spaces for the connecting map in \Cref{thm:unique}} We are now ready to state our main technical result, wherein the neck-stretching is used to analyze precisely what curves appear in the computation of the connecting map $\delta_{\partial}\colon SH_+^{*,<D}(W,\Lambda) \to H^*(Y,\Lambda)[1]$ under the hypotheses of \Cref{thm:unique}. 
Here, $\Lambda$ is the Novikov field, which appears due to the fact that we consider general strong fillings. 
Moreover, despite the fact that we will consider strong fillings of the connected sum, we will use a cobordism trick to reduce to the case that $M$ is Liouville (in fact, Weinstein) fillable.

\begin{proposition}\label{prop:curve}
Let $(Y,\xi)$ denote be the contact connected sum $M\#\partial W_0$ given in the statement of \Cref{thm:unique} with the additional assumption that $M$ is Liouville fillable and the Liouville filling has trivial first Chern class in the case (a) of \Cref{thm:unique}, and  $\beta$ a \emph{non-trivial} class in the image of $H^n(W_0;\Q)\to H^n(\partial W_0;\Q)$.
 Then there are
\begin{itemize}
    \item a contact form $\alpha$ for $\xi$,
    \item an admissible pair $(H,J)$ of a Hamiltonian and an almost complex structure on the completion $\widehat{Y}$ as in \S \ref{SS:4.4},
    \item a Morse function $g$ on $Y$,
    \item  a $\Q$-linear combination $\sum_i a_i x_i$ of non-constant Hamiltonian orbits of $H$,  which are of Conley-Zehnder index $2$ in the case that $c_1(M\#\partial W_0)=0$,
\end{itemize}  

so that the following properties hold:
	\begin{enumerate}
        \item\label{c1} $\sum_i a_i \,\#\overline{\cM}_{Y,H}(x_i,y,\emptyset)=0$, for every non-constant Hamiltonian orbit  $y$; to be precise, the moduli spaces of the type $\overline{\cM}_{Y,H}(x_i,y,\emptyset)$ that are of expected dimension at most $0$ are cut out transversely, and the sum $\sum a_i \,\#\overline{\cM}_{Y,H}(x_i,y,\emptyset)$ for those with expected dimension $0$ vanishes;
          \vspace{3pt}
        \item\label{c2} $\overline{\cM}_{Y,H}(x_i,y,\Gamma)=\emptyset$ for all $i$, $y$ and $\Gamma \ne \emptyset$ multiset of $\alpha$-Reeb orbits;
          \vspace{3pt}
        \item\label{c3}
        For any index $n$ critical point $p$ of $g$, we have
        $\overline{\cM}_{Y,H}(x_i,S_{p},\Gamma)=\emptyset$ for $\Gamma \ne \emptyset$.
        Moreover, $\sum_{i} \sum_{p} a_i\,\#\overline{\cM}_{Y,H}(x_i,S_{p},\emptyset) \, p$  
        represents $0\oplus\beta$ in $H^n(M;\Q)\oplus H^n(\partial W_0;\Q)=H^n(Y;\Q)$ in the Morse cochain complex of $g$, where the inner sum is over all index $n$ critical points $p$ of $g$ and $S_p$ denotes the stable manifold of $p$ with respect to the gradient $\nabla g$.
	\end{enumerate}
\end{proposition}

The above result is what guarantees that cohomology classes in the image of $H^n(W_0;\Q)\to H^n(\partial W_0;\Q)$ are ``detected'' via the connecting map $\delta_\partial$ and is independent of the strong filling.
Indeed, the significance of the properties above is the following: 
\eqref{c1} combined with \eqref{c2} will tell us that $\sum_i a_i x_i$ is closed with respect to the Floer differential and defines an element in the positive symplectic cohomology with respect to \emph{any} symplectic filling;
\eqref{c3} will imply that the degree-$n$ term in $\delta_{\partial}(\sum a_ix_i)$ is exactly $0\oplus \beta$.

\bigskip

We postpone the proof of \Cref{prop:curve}, and use it first to prove \Cref{thm:unique}.

\begin{proof}[Proof of \Cref{thm:unique}]
As in the statement, assume $W$ is a strong filling of $Y$.
By the universal coefficient theorem, it is equivalent to proving the dual statement that $H^n(W;\Q)\to H^{n}(Y;\Q)$ is surjective onto $0\oplus\image(H^n(W_0;\Q)\to H^n(\partial W_0;\Q))$.
By \cite[Theorem 1.7]{MR4100126}, there is a Weinstein cobordism from $M$ to $M'$, such that $M'$ is Weinstein fillable. 
In the case where $\dim M =5$, according to \cite[Proposition 3.4]{BCS2}, we can moreover assume that the Weinstein cobordism has vanishing first Chern class if $M$ does. In particular, $M'$ has a Weinstein filling with vanishing first Chern class.
Then, we have a Weinstein cobordism from $Y$ to $Y':=M'\# \partial W_0$, yielding a strong filling $W'$ of $Y'$. 
Observe that the desired claim follows from the corresponding statement for $Y' = \partial W'$, since the surgeries do not affect the $\partial W_0$-factor.
In other words, it suffices to prove the conclusion of \Cref{thm:unique} under the additional assumption that $M$ is Liouville fillable.

Let $\beta\in \image(H^n(W_0;\Q)\to H^n(\partial W_0;\Q))$; we need to prove that $0\oplus \beta$ is in the image of $H^n(W;\Q)\to H^{n}(Y;\Q)$.
As we can assume that $M$ has a Liouville filling $F$, we can apply \Cref{prop:curve} and we let $\sum a_ix_i$ be the resulting $\Q$-linear combination of non-constant periodic orbits of the Hamiltonian.
Viewing $H$ as a Hamiltonian on $\widehat{W}$ by extending it as zero over the filling,
we can then consider the sum $\sum a_i q^{\mathcal{A}_H(x_i)}x_i$ as an element of the Floer complex of $H$ (considered as a Hamiltonian on $\widehat{W}$), with coefficients in the Novikov field\footnote{Symplectic cohomology of strong fillings requires Novikov coefficients, and different orders of applying direct limit and completion in the construction yield different theories, c.f.\ \cite{zbMATH06864027}. As we only use symplectic cohomology of a finite slope, such subtleties do not arise.}
$\Lambda$, where $q$ is the formal variable in $\Lambda$ and the symplectic action $\mathcal{A}_H(x_i)$ is computed using the Liouville form in the cylindrical end containing $x_i$.

We now claim that, for a sufficiently stretched almost complex structure, the chain $\sum a_i q^{\mathcal{A}_H(x_i)}x_i$ is in fact closed w.r.t.\ the Floer differential so that the sum represents a class in $SH_+^*(W;\Lambda)$. 
Indeed, using neck-stretching along the contact boundary, by \eqref{c1} and \eqref{c2} of \Cref{prop:curve}, one can arrange that all the Floer cylinders used in the differential are contained in the positive end of the symplectization, whose energy is given by the difference of the respective symplectic actions $\mathcal{A}_H$,  which is well-defined on the positive end of the completion. 
Therefore, using \eqref{c1} of \Cref{prop:curve}, we have that
\begin{equation*}
\begin{split}
\delta\left(\sum_i a_i q^{\mathcal{A}_H(x_i)}x_i\right)
=&
\sum_y \sum_i a_i \#\overline{\mathcal{M}}_{Y,H}(x_i,y,\emptyset)
q^{\mathcal{A}_H(x_i)+\mathcal{A}_H(y)-\mathcal{A}_H(x_i)} y\\
=&
\sum_y \underbrace{\left[\sum_i a_i \#\overline{\mathcal{M}}_{Y,H}(x_i,y,\emptyset)\right]}_{ = \, 0}
q^{\mathcal{A}_H(y)} y=0.
\end{split}
\end{equation*}

Similarly, \eqref{c3} of \Cref{prop:curve} implies that $\sum a_i q^{\mathcal{A}_H(x_i)}x_i$ is mapped to $0\oplus \beta$ under the composition 

\[
\Pi_n\circ \delta_\partial\colon SH^*_+(W;\Lambda)\to H^{*+1}(W;\Lambda)\to H^{*+1}(Y;\Lambda)\to H^n(Y;\Lambda) \, ,
\]
where $\Pi_n\colon H^{*+1}(Y;\Lambda)\to H^n(Y;\Lambda)$ denotes the natural projection and $0\oplus\beta\in H^{*}(Y;\Q)$ is considered as an element in $H^*(Y;\Lambda)\cong H^*(Y;\Q)\otimes_\Q \Lambda$.
Hence, $0\oplus \beta$ also lies in the image of $H^{n}(W;\Q)\to H^{n}(Y;\Q)$, as desired.
\end{proof}

\bigskip

\subsection*{Proof of \Cref{prop:curve} when $c_1(Y)=0$}
Roughly speaking, the proof is structured as follows:
We first describe how to find the contact form $\alpha$ on $Y$, an admissible pair $(H,J)$, and the $\Q$-linear combination $\sum_i a_i x_i$ as in the statement.
The auxiliary Morse function is arbitrary.
Then, we proceed to prove all the properties of these objects as claimed in the statement. 
For this, the strategy of stretching the neck will be used repeatedly. We now give the details.

\begin{proof}[Proof of \Cref{prop:curve} for $c_1(Y)=0$] 
By the surgery formulae in \cite[Theorem 1.11]{Cie02b} and \cite[Theorem 5.6]{ BourEkEl}, or alternatively \cite[Theorem 3.2]{zbMATH06864340}, the symplectic cohomology $SH^*(W_0)$ is trivial so that the connecting map $SH^*_+(W_0)\to H^*(W_0)[1]$ is an isomorphism for the flexible filling $W_0$. 
Moreover, according to \cite[Theorem 1.9]{Laz20}, there is a sequence of contact forms $\alpha_1>\alpha_2>\ldots$ and positive numbers $D_i\to \infty$, such that all contractible Reeb orbits $\alpha_i$ with period smaller than $D_i$ are non-degenerate and have Conley-Zehnder index at least $1$\footnote{ \cite[Theorem 1.9]{Laz20} states that the SFT degrees are at least $1$, but the proof actually shows that Conley-Zehnder indices are at least $1$ by \cite[Propositions 4.3, 4.4]{Laz20}, see \cite[Example 2.6 (1)]{zbMATH07706508}.}. As the direct limit of symplectic cohomology of $W_0$ with boundary contact form $\alpha_i$ and period upper threshold $D_i$, through the Viterbo transfer maps and continuation maps, is the symplectic cohomology of $W_0$, see e.g.\ \cite[Proof of Proposition 3.8]{Laz20}, we may assume there exists a contact form $\alpha_0$ on $\partial W_0$ and a positive real number $D$ such that all Reeb orbits of period $<D$ are non-degenerate and have Conley-Zehnder indices at least $1$, with the additional property that $SH_+^{*,<D}(W_0)\to H^*(W_0)[1]$ is surjective, or equivalently $H^*(W_0)\to SH^{*,<D}(W_0)$ is zero.

We now consider the connected sum of the strict contact manifolds $(\partial W_0,\alpha_0)$ and $(M, K\alpha_M)$, for some $K\gg 0$, and contact form $\alpha_M$ on $M$. 
According to \Cref{lemma:connected_sum}, as long as $K\gg D$, we can assume that all the Reeb orbits of period $<D$ on the connected sum $(M\# \partial W_0,\alpha)$ are those in the connected-summand $\partial W_0$, plus some of the multiple covers $\gamma_{h,i}^k$ of simple Reeb orbit $\gamma_{h,i}$ on the co-core of the $1$-handle. Let now $F$ be the Liouville filling of $M$ with trivial first Chern class. 
Using the explicit contact form on the boundary $M\#\partial W_0$ of the boundary connect sum $F\natural W_0$ as described above, \Cref{prop:symp_sum} implies that
$$SH^{n,<D}(F\natural W_0)=SH^{n, <D}(F)\oplus SH^{n,<D}(W_0)$$
via the Viterbo transfer map. As a consequence, we know that $H^n(W_0) \to SH^{n,<D}(W_0) \to SH^{n,<D}(F\natural W_0)$ is zero.

Now,  the long exact sequence for filtered positive symplectic cohomology implies that $SH^{n-1,<D}_+(F\natural W_0)\to H^{n}(F\natural W_0)$ is surjective onto $0\oplus H^n(W_0)$. 
Therefore, we have that 
\[\delta_{\partial}:SH^{n-1,<D}_+(F\natural W_0)\to H^{n}(F\natural W_0) \to H^n(M\# \partial W_0)\] 
is surjective onto $0\oplus \image(H^n(W_0)\to H^n(\partial W_0))$. 
Since $\beta \in \image(H^n(W_0)\to H^n(\partial W_0))$, this then implies that there are $x_i$ Hamiltonian orbits with Conley-Zehnder index $2$ and a linear combination $\sum a_ix_i$ representing an element in the positive symplectic cohomology (i.e.\ in particular closed w.r.t.\ the differential) so that  $\delta_{\partial}(\sum a_ix_i)=\beta$.

Now we apply neck-stretching to the moduli spaces defining the differential and $\delta_{\partial}$ in order to deduce the desired properties. 
The expected dimension for $\overline{\cM}_{Y,H}(x_i,y,\Gamma)$ is 
$$\mu_{CZ}(x_i)-\mu_{CZ}(y)-1-\sum_{\gamma\in \Gamma}(\mu_{CZ}(\gamma)+n-2)\le 2-1-1-\sum_{\gamma\in \Gamma}(n-1)= -(n-1) \, \# \Gamma$$
which is negative if $\Gamma \ne \emptyset$; hence \eqref{c2} follows. 
Similarly, the expected dimension of $\overline{\cM}_{Y,H}(x_i,S_{p},\Gamma)$ is
$$0-\sum_{\gamma\in \Gamma}(\mu_{CZ}(\gamma)+n-2)\leq -(n-1) \, \#\Gamma \, ,$$
which is again negative if $\Gamma\neq \emptyset$,
hence the first claim of \eqref{c3} holds.

When $y$ has Conley-Zehnder index larger than $1$, \eqref{c1} holds since the expected dimension in this case is negative. 
When $y$ has Conley-Zehnder index $1$, by \eqref{c2}, for a sufficiently stretched almost complex structure, the Floer cylinders that are counted when defining $SH^*_+(F\natural W_0)$ are completely contained in the positive end of $\widehat{F\natural W_0}$, i.e.\ the corresponding moduli space can be identified with $\overline{\cM}_{Y,H}(x_i,y,\emptyset)$. Since the above properties on filtered symplectic cohomology do not depend on the choice of almost complex structure, we may assume $\sum a_ix_i$ is closed for a sufficiently stretched almost complex structure. This implies \eqref{c1}.

For the second claim of \eqref{c3}, we apply neck-stretching to the moduli space contributing to $\delta_{\partial}(\sum a_ix_i)$. 
Since we have $\overline{\cM}_{Y,H}(y,S_{p},\Gamma)=\emptyset$ for any $y$ with Conley-Zehnder index at most $1$ by dimension counting, there cannot be other degeneration in the neck-stretching, and so for any sufficiently stretched almost complex structure the curve contributing to $\delta_{\delta}(\sum a_i x_i)$ are completely contained in the positive end of $\widehat{F\natural W_0}$, which can be identified with $\cM_{Y,H}(x_i,S_p,\emptyset)$. Therefore we have 
$$
\sum_{i} \sum_{p} a_i\,\#\overline{\cM}_{Y,H}(x_i,S_{p},\emptyset) p=0\oplus\beta\in H^n(M\#\partial W_0).$$
\end{proof}

\subsection{The subcritical case} As already pointed out above, the fact that $c_1(Y)$ need not vanish means that the Conley-Zehnder indices are not well-defined.
In this case, we make up for this loss with an energy bound (in \eqref{eqn:area} below), which is derived leveraging the subcriticality of $W_0$ in a substantial way.

\subsection*{Geometric set-up} We start by describing an explicit choice of contact forms on $ \partial (V\times \D^2)$ as in \cite[\S 2.1]{product}.
More precisely, one first rounds the corner of 
\begin{equation}
\label{eqn:product_corner}
\left(
\, 
\partial(V\times \D^2) = V\times \SSS^1\cup_{\partial V \times \SSS^1} 
\partial V \times \D^2 \, ,
\; \lambda_V+\frac{r^2}{2\pi} \mathrm{d} \theta\,\right)\, ,
\end{equation}
and
the desired contact form $\alpha$ on $\partial(V\times \D^2)$ is then a perturbation of $\lambda_V+\frac{r^2}{2\pi} \mathrm{d} \theta$ supported on the mapping torus region $V\times S^1$.
As explained in detail in \cite[\S 2.1]{product}, this perturbation is achieved via an auxiliary Morse function $f$ on $V$, satisfying the following properties: 
\begin{itemize}
    \item it is ``self-indexing'', by which we mean that $f(p)>f(q)$ if and only if $\ind (p) > \ind (q)$ for every pair of critical points $p,q$ (in particular, all critical points of the same index have the same image under $f$);

    \item $f(p)$ is approximately $1$ for every critical point $p$ of $f$;
    \item $f\to 1$ when approaching the boundary of $V$, whose precise meaning will be clear from the discussion below.  
\end{itemize}
We use such $f$ to smooth the corners of the hypersurface $\partial(V\times \D^2)$ inside the completion $(\widehat{V}\times \CC, \widehat{\lambda}_V+\frac{r^2}{2\pi}\rd \theta)$ by bumping up the $V\times \SSS^1$ using the graph of $r^2=1/f$. It is clear that we can choose $f$ such that $\rd f$ is $C^1$-small outside a collar neighborhood of $\partial V$ in $V$ and $f$ only depends on the Liouville coordinate near $\partial V$, whose derivatives diverge when approaching the boundary, to make the perturbed hypersurface smooth and of contact type. For an $f$ chosen as above, the perturbed contact form $\alpha$ is given by $\lambda_V+\frac{r^2}{2\pi}\mathrm{d}\theta$ on $\partial V\times \D^2$, and by $\lambda_V +\frac{1}{2\pi f} \mathrm{d} \theta$ on $V\times \SSS^1 = \partial(V\times \D^2)\, \setminus \, \partial V \times \D^2 $.
These are explained in detail in \cite[\S 2.1]{product}. 
Moreover, it satisfies the following properties.
\begin{enumerate}
	\item Each critical point $p$ of $f$ corresponds to a simple non-degenerate Reeb orbit $\gamma_p$, which is the circle over $p$ in the region $V\times \SSS^1\subset \partial(V\times \D^2)$. 
    In particular, the Reeb orbits $\gamma_{p_i}$'s wind around the binding $\partial V \times \{0\} \subset \partial V\times \D^2$ exactly once. 
	\item The period of $\gamma_p$ is $1/f(p)$ (which is approximately $1$ by our choice of $f$), and hence $\gamma_p$ has longer period than $\gamma_q$ if and only if $f(p)<f(q)$.
    \item The set of all Reeb orbits with period $<2$ is just $\{\gamma_p \; \vert \; p\in \mathrm{Crit}(f)\}$, see \cite[Proposition 2.2]{product}. 
\end{enumerate}	

We next explain how to construct an admissible almost complex structure to ensure the energy bound in \eqref{eqn:area} in the following lemma, which we will then use in the proof of \Cref{prop:curve} in the subcritical case.

\begin{lemma}
    \label{lem:contact_form_alm_complex_str_subcritical_case}
    Let $V$ be a $2n$-dimensional Weinstein domain for $n\geq 2$. 
    Consider also $W_0=V\times \DD^2$ with its natural product symplectic structure and $(M,\xi)$ any contact manifold. 

    Then, there are 
    \begin{itemize}
        \item a Morse function $f$ on $V$ as in the definition of the contact form on $\partial(V\times \DD^2)$ above with maximal Morse index $n$,
        \item a contact form $\alpha$ on $Y=\partial W_0\# M$,
        \item an admissible almost complex structure $J$ on the symplectization $\RR_s \times Y$ 
    \end{itemize}
    satisfying the following properties:
    \begin{enumerate}
        \item\label{item:lemma_connected_sum_reeb_orbit_low_action} the Reeb orbits of $\alpha$ that have period $<2$ are
        \begin{itemize}
            \item either from the set $\{\gamma_p \mid p \in \mathrm{Crit}(f)\}$, where each $\gamma_p$ has period $1/f(p)$ and has linking number $1$ in $Y$ with the binding $\partial V$ of the natural trivial open book on $\partial(V\times \DD^2)$,
            \item or a multiple cover of a simple Reeb orbit $\gamma_{h,i}$ in the belt sphere of the $1$-handle in the connected sum, which has zero linking number with the binding;
        \end{itemize}
        \item\label{item:lemma_connected_sum_reeb_orbit_handle} the simple Reeb orbit $\gamma_{h,i}$ satisfies the action constraint
    \begin{equation}\label{eqn:condition}
        \int \gamma_{h,i}^*\alpha > \frac{1}{f(p)}-1\, ,
        \tag{*}
    \end{equation}
    for all critical points $p$ of $f$ that are \emph{not} the minimum and $1\le i \le n$;
    
    \item\label{item:lemma_connected_sum_almost_complex_str} all the maps $u$ defined on a subset of $\RR\times \mathbb{S}^1$ that are contained in $\RR_s\times \partial V\times \DD^2$ and satisfy the Floer equation with respect to a Hamiltonian depending only on $s$, project via the natural projection $\pi_{\DD}\colon \RR_s\times \partial V\times \DD^2\rightarrow \DD^2$ to a holomorphic curve (with respect to the standard $i$ on $\DD^2$), and 
    we have 
    \begin{equation}\label{eqn:area}
        \int u^*\rd\alpha \ge \int (\pi_{\D}u)^*\rd \left(\frac{r^2}{2\pi}\rd \theta\right)\ge 0.
        \tag{**}
    \end{equation}
    \end{enumerate}
    
\end{lemma}
Recall that, according to \cite{Cie02}, any subcritical Weinstein domain $W_0$ is Liouville homotopic to the symplectic product $V\times \DD^2$ for some Weinstein manifold $V$.
In particular, the assumptions in \Cref{lem:contact_form_alm_complex_str_subcritical_case} do indeed correspond to the hypothesis of \Cref{thm:unique} (in the subcritical $W_0$ case).

We now give a proof of this preparatory lemma.

\begin{proof}[Proof of \Cref{lem:contact_form_alm_complex_str_subcritical_case}]
The connected sum $M\# \partial(V\times\D^2)$ is given by contact surgery of index $1$ on the disjoint union $M\sqcup \partial(V\times\D^2)$; equivalently, it is the convex contact boundary of a Weinstein $1$-handle attachment on the disjoint union of  $M$ with $\partial(V\times \D^2)$. This surgery (or handle attachment) can be performed inside the subset $V\times \SSS^1\subset \partial(V\times\D^2)$, away from every critical point of $f$ times the $\SSS^1$ factor.
We first arrange property \eqref{item:lemma_connected_sum_reeb_orbit_handle}. 
For this, we start with a function $f_0$ on $V$ satisfying all the properties (of the function that was then denoted $f$) described right before the statement of this lemma. 
We now explain in steps how such functions can be modified in order to arrange property \eqref{item:lemma_connected_sum_reeb_orbit_handle}. 
\begin{enumerate}
    \item Let $x_0$ be the unique minimum point of $f_0$. We may assume the Liouville form on $\lambda_{V}$ on a ball $B$ around $x_0$ is the standard radical Liouville form on $\CC^n$. Then $(V\backslash B,\lambda_V)$ is a Liouville cobordism from the standard contact sphere with the standard contact form.   
    \item We may assume the $f_0$ near $\partial B$ only depends on the radius. More precisely, those can be achieved by first finding a chart near $x_0$ with $f_0$ depending only on the radius, then changing the Liouville form $\lambda_V$ to the standard one via a compactly supported Liouville homotopy with the same symplectic form. 
    \item We attach the handle to $(B\times \SSS^1, \lambda_V+\frac{1}{2\pi f_0}\rd \theta)$. In particular, the size of the handle is fixed at this stage.
    \item On $V\backslash B$, we modify $f_0$ to $f_{\epsilon}=1+\epsilon(1-f_0)$. As $\int\gamma_{h,i}^*\alpha$ is fixed already, it is clear that for $0<\epsilon\ll 1$, property \eqref{item:lemma_connected_sum_reeb_orbit_handle} holds for critical points of $f_{\epsilon}$ on $V\backslash B$.
    \item $f_{\epsilon}$ and $f_0$ do not match on $\partial B$, where $f_{\epsilon}$ has a larger value and smaller derivatives. However, we can scale up the Liouville form  $\lambda_V$ on $V\times B$ to $K\lambda_{V}$ for $K\gg 0$ and insert a large trivial cobordism $\partial B\times (1,K)$ of $\partial B$ in between $V\times B$ and $B$. The glued domain is denoted by $V'$. It is diffeomorphic to $V$, one can also view this procedure as changing the Liouville form $\lambda_V$ on the fixed $V$. Then we have $f_\epsilon$ defined on $V\backslash B$ part and $f_0$ is defined on $B$ part, we can use a function only depending on the $r$-coordinate of $\partial \times (1,K)_r$ to interpolate between $f_0,f_{\epsilon}$. The glued function on $V'$ is denoted by $f$.
    \item We need to argue that $f$ on $V'$ has the same property discussed before the lemma. The most important property is regarding the short Reeb orbits. By the discussion in \cite[\S 2.1]{product} and \cite[Proposition 6.4]{Zho21}, this essentially requires the Hamiltonian vector field of $1/f$ w.r.t.\ the symplectic form $\rd\lambda_V$ \footnote{\cite[Proposition 6.4]{Zho21} uses $\rd(\lambda_V/f)$ and a horizontal lift. In our case here, it is straightforward to compute that the formula for Reeb vector in \cite[Proposition 6.4]{Zho21} is $(1/f-\lambda_V(X_{1/f}))^{-1}(\partial_\theta-X_{1/f})$, where $X_{1/f}$ is tha Hamiltonian vector field of $1/f$ using $\rd \lambda_V$.} to have no non-constant orbits with short periods. This is why we need the $C^1$-smallness of $\rd f$ outside of a collar neighborhood of $\partial V$. The procedure of inserting a large trivial cobordism, i.e.\ blowing up the symplectic form, will scale down the Hamiltonian vector field even further. Therefore, it will not break the property on short Reeb orbits. In the following, we will return to the notation of $(V,\lambda_V)$, which is considered as having gone through the modification above. 
\end{enumerate}
By scaling up the contact form on $M$ and \Cref{lemma:connected_sum},  we can arrange that the Reeb orbits on $M\#\partial(V\times \D)$ of period $<2$ are just the $\gamma_p$'s and (some of) the multiple covers of simple Reeb orbits $\gamma_{h,i}$ on the belt sphere of the handle, thus proving property \eqref{item:lemma_connected_sum_reeb_orbit_low_action} modulo the claim on linking numbers. As the binding $V\times \{0\}$ is null homologous, the linking number can be defined using any bounding disk of the Reeb orbit. For $\gamma_p$, we can use the natural disk from the $\DD^2$ factor, and for $\gamma_{h,i}^k$, we can use a bounding disk contained in the handle region to deduce the claim concerning the linking number.
\medskip

All that is left to do is describe the desired almost complex structure $J$, and prove property \eqref{item:lemma_connected_sum_almost_complex_str}. On the neighborhood $(\partial V \times \D^2,\alpha=\lambda_V+\frac{r^2}{2\pi}\rd \theta)$ of the binding, the contact structure is given by
\[
\xi_{\partial V }\oplus \la \partial_x+\frac{y}{2\pi}R_{\lambda_V},\partial_y-\frac{x}{2\pi}R_{\lambda_V} \ra,
\]
where $\xi_{\partial V}\subset T\partial V$ is the contact structure $\ker \lambda_V$, $R_{\lambda_V}$ is the Reeb vector field of $(\partial V, \lambda_V)$ and $x,y$ are coordinates on $\D^2$. 
Moreover, the Reeb vector field $R$ of $\alpha_0$ on the neighborhood $\partial V \times \D^2$ of the binding is just $R_{\lambda_V}$. 
We then use an almost complex structure $J$ on $\RR_s \times \partial V \times \D^2 $ satisfying the following properties:
\[
J:\xi_{\partial V} \to \xi_{\partial V} \text{ is compatible with } \mathrm{d} \lambda_V,
\quad J(\partial_x+\frac{y}{2\pi}R_{\lambda_V})=\partial_y-\frac{x}{2\pi}R_{\lambda_V}
\, \text{ and } \, 
J(R_{\lambda_V})=-\partial_s \, .
\] 
Note in particular that $\pi_{\D}$ is $(J,i)$ holomorphic, where $i$ is the standard complex structure on $\DD^2$. 

Moreover, if we have a Hamiltonian function $H$ on $\RR_s\times \partial V\times \DD^2$ that only depends on the $s$-coordinate, its Hamiltonian vector field is parallel to $R_{\lambda_V}$. 
If $u$ solves the Floer equation in $\RR_s\times \partial V \times \D^2$, namely
$$\partial_su+J(\partial_tu-X_H)=0\, ,$$
then $\pi_{\D} u$ is a holomorphic map to $\D$. Note that 
$$u^*\rd \alpha(\partial_su,\partial_tu) = \rd\alpha (\partial_s u, J\partial_su)=\rd\lambda_V(\pi_{\xi_{\partial V}}\partial_s u,\pi_{\xi_{\partial V}}J\partial_su)+\rd\left(\frac{r^2}{2\pi}\rd \theta \right)(\pi_{\DD}\partial_su,\pi_{\DD} J\partial_su)$$
where $\pi_{\xi_{\partial_V}}$ is the natural projection to $\xi_{\partial V}$ and $\pi_{\DD}$, by a bit abuse of notation, also denotes the natural projection to the $\DD^2$-direction on the tangent bundle. 
Now, we can compute
$$ \rd\lambda_V(\pi_{\xi_{\partial V}}\partial_s u,\pi_{\xi_{\partial V}}J\partial_su)=\rd\lambda_V(\pi_{\xi_{\partial V}}\partial_s u,J\pi_{\xi_{\partial V}}\partial_su)\ge 0$$
and 
\begin{eqnarray*}
    \rd\left(\frac{r^2}{2\pi}\rd \theta \right)(\pi_{\DD}\partial_su,\pi_{\DD} J\partial_su) & = & \rd\left(\frac{r^2}{2\pi}\rd \theta \right)(\pi_{\DD}\partial_su,i\pi_{\DD} \partial_su) \\
    & = & \rd\left(\frac{r^2}{2\pi}\rd \theta \right)(\partial_s(\pi_{\DD}\circ u),i\partial_s(\pi_{\DD}\circ u)) \\
    & = & \rd\left(\frac{r^2}{2\pi}\rd \theta \right)(\partial_s(\pi_{\DD}\circ u),\partial_t(\pi_{\DD}\circ u))\\
    & = & (\pi_{\DD}\circ u)^* \rd\left(\frac{r^2}{2\pi}\rd \theta \right)(\partial_su,\partial_tu) \, .
\end{eqnarray*}
Therefore, \eqref{eqn:area} holds.
This concludes the proof of property \eqref{item:lemma_connected_sum_almost_complex_str}, and hence of  \Cref{lem:contact_form_alm_complex_str_subcritical_case}.
\end{proof}

\begin{proof}[Proof of \Cref{prop:curve} for $W_0$ subcritical]
The contact form $\alpha$ (depending on a function $f$ on $V$ as above) and almost complex structure $J$ are exactly those given by \Cref{lem:contact_form_alm_complex_str_subcritical_case}; 
we also use here $\gamma_p$ and $\gamma_{h,i}$ as given by said lemma. We also fix a Morse function $g$ on $\partial(V\times \DD^2)$.
We now proceed to describe the desired Hamiltonian, $\Q$-linear combination, and then to prove the desired properties.

\medskip

We start by describing the Hamiltonian.
First, on the filling $W_0=V\times \D^2$ of $\partial(V\times \D^2)$, one can consider a Hamiltonian $H$ which is linear and with slope $2$ outside a cylindrical neighborhood of the boundary, which vanishes on a slightly shrinked version of the filling, and which interpolates between these two behaviors along the cylindrical end, where it depends only on the Liouville coordinate. To utilise the positivity of contact energy in \eqref{eqn:contact}, we will use a Hamiltonian that only depends on $r$. This requires the cascades model to define the symplectic cohomology to deal with Morse-Bott families from the $S^1$-symmetry following \cite{BODuke} or see \cite[\S 2.2]{product}. From each one of the Reeb orbits $\gamma_p$'s from \Cref{lem:contact_form_alm_complex_str_subcritical_case}, with $p$ a critical point of the Morse function $f$, we associate with two generators $\check{\gamma}_p,\hat{\gamma}_p$ for symplectic cohomology. The moduli spaces in the cascades model, analogous to those in \S \ref{SS:4.4} were described in \cite[\S 2.5]{product}. In practice, one can neglect the difference between cascades and non-cascades versions of those in \S \ref{SS:4.4} as they enjoy the same properties (virtual dimensions, energy, etc.). 

As far as the desired $\Q$-linear combination is concerned, we also point out that, for every 
\[
\beta \in \image(H^n(V\times \D) \to H^n(\partial(V\times \D)))\cong H^n(V) \, ,
\]
we can first represent $\beta=\sum a_i p_i$ in Morse cohomology, with $\ind (p_i)=n$. Then, according to \cite[Proposition 2.6]{product}, $\sum a_i\check{\gamma}_{p_i}$ in $SH^*_+$ satisfies $\delta_{\partial}(\sum a_i\check{\gamma}_{p_i})=\beta$.

We now look at the connected sum of $\partial W_0$ with $M$ described before, and use the special contact form $\alpha$ from \Cref{lem:contact_form_alm_complex_str_subcritical_case} as described above.
Recall in particular that the Reeb orbits on $M\#\partial(V\times \D)$ of period $<2$ are just the $\gamma_p$'s and (some of) the multiple covers of simple Reeb orbits $\gamma_{h,j}$ in the belt sphere of the handle, and that the inequality \eqref{eqn:condition} holds for any critical point $p_i$ which is not the minimum of $f$. Unlike the argument in the flexible case, we will use a Hamiltonian on the Liouville completion of $(V\times \DD^2)\natural F$ that has slope $2$ outside a cylindrical neighborhood of the boundary and vanishes on a slightly smaller copy of the filling.

Note that, since the $p_i$'s in the above identity $\beta=\sum a_i p_i$ all have index $n$, the associated $\gamma_{p_i}$'s have the minimum period among all the Reeb orbits of the form $\gamma_{p}$.
In particular, \eqref{c1}, \eqref{c2}, and the first claim of \eqref{c3} will directly follow from action considerations, provided we can rule out the possibility of the $\gamma_{h,j}^k$'s, or the corresponding Hamiltonian orbits, appearing as negative asymptotics. 
To do this, we use an argument inspired by \cite[Lemma 5.5]{CDvK16}, which leverages the energy lower bound described in \eqref{eqn:area} of \Cref{lem:contact_form_alm_complex_str_subcritical_case} in the binding region. 
(In fact, our situation is simpler than that in \cite{CDvK16}, as we are working with the trivial open book $\partial(V\times \D)$.)

As the only Reeb orbits of action $<2$ on the connect sum have been arranged to be the ones from $\partial(V\times \D^2)$ and the iterates of the Reeb orbit on the belt sphere of the handle, for the moduli spaces in the conclusion of the proposition, it is enough to look at $y=\overline{\gamma}_{h,j}^k\in \{\check{\gamma}_{h,j}^k,\hat{\gamma}_{h,j}^k\}$ and $\Gamma$ a multiset of Reeb orbits taken from the $\{\gamma_{h,j}^k\}_{k,j}$.
Then, for any curve $u$ in $\overline{\cM}_{Y,H}(\check{\gamma}_{p_i},\overline{\gamma}_{h,j}^k,\Gamma)$, we have 
\begin{equation}\label{eqn:contact}
     \int_{S^1} \check{\gamma}_{p_i}^*\alpha-
     \int_{S^1}(\gamma_{h,j}^k)^*\alpha -
     \sum_{\gamma \in \Gamma} \int_{S^1}\gamma^*\alpha
     \; \ge \;
     \int_{u^{-1}(\RR_t\times\partial V \times \D^2)}
     u^*\mathrm{d} \alpha
     \; \ge \;
     \int_{u^{-1}(\RR_t\times\partial V \times \D^2)}
     (\pi_{\D}u)^*\mathrm{d}\left(\frac{r^2}{2\pi}\mathrm{d} \theta\right)=1 
     \, ,
\end{equation}
where the first inequality follows from Stokes' theorem as $u^*\rd \alpha \geq 0$ on the whole curve as $H$ only depends on $r$ in this region, the second inequality from \eqref{eqn:area}, and the last equality from the fact that the linking number of the $\gamma_{p_i}$'s and of the $\gamma_{h,j}^k$'s around the binding $\partial V$ are $1$ and $0$ respectively. 
Hence we arrive at a contradiction with the inequality \eqref{eqn:condition}, provided such a curve $u$ in $\overline{\cM}_{Y,H}(\check{\gamma}_{p_i},\overline{\gamma}_{h,j}^k,\Gamma)$ exists. 
This establishes \eqref{c1}, \eqref{c2} and the first claim of \eqref{c3}. 

Lastly, using the Viterbo transfer map we see that, in the connected sum, $\sum a_i\check{\gamma}_{p_i}$ on $M\#\partial(V\times \D)$ is mapped to $\sum a_i\check{\gamma}_{p_i}$ on $\partial(V\times \D)$. This can be seen using the action filtration as $\gamma_{p_i}$ has the minimal period from those $\gamma_p$, the transfer map from $\check{\gamma}_{p_i}$ are built from reparameterization of the trivial cylinder, as explained in \cite[\S 4.2, (4.3)]{zhou2023contact}.
Hence $\delta_{\partial}(\sum a_i\check{\gamma}_{p_i})=0\oplus \beta$ for the Liouville filling $F\natural (V\times \D)$. 
Combining with the just proved emptiness of the moduli spaces of the points \eqref{c1}, \eqref{c2}, and \eqref{c3} in the statement, we then get the second claim of \eqref{c3} by a neck-stretching argument by applying the proof of \Cref{prop:curve} in the previous case {\em mutatis mutandis}.
\end{proof}

%% file: s5.tex
\section{Non-fillable contact structures}\label{sec:geometric_construction}

\subsection{Tight non-fillable structures on spheres}
We first construct tight, non-fillable contact structures on $\SSS^{2n+1}$, for $n\geq 2$ to prove the following stronger version of \Cref{thm:tightnonfill}. 
\begin{theorem}\label{thm:tightnonfill_alge}
	For every $n \ge 2$, the sphere $\mathbb S^{2n+1}$ admits an \textbf{algebraically} tight, non-strongly fillable contact structure that is homotopically standard.
\end{theorem}

Note that \Cref{thm:tightnonfill_alge} directly gives \Cref{thm:tightnonfill} because of the general fact that algebraic tightness implies tightness, as recalled in \Cref{prop:property} (1).

\begin{proof}

\textbf{Step 1: the geometric construction.} We consider the Bourgeois contact structure coming from the following choice of open book on the sphere $\SSS^{2n-1}$.
For $n\geq 3$, we chose the open book coming from the $A_{1}$-singularity, with binding given by $B=\Sigma_{n-1}(2,2,\cdots,2) \cong S^*\SSS^{n-1}$; this is just the positive stabilization $\SSS^{2n-1}=OB(D^*\SSS^{n-1},\tau)$ of the standard open book $OB(\D^{2n-2},\Id)$, where $\tau$ is the Dehn-Seidel twist, as defined in detail in \cite{Sei03} following the idea in \cite{Arn95}.
For $n=2$, we chose an open book $OB(\Sigma, \psi)$ for the standard contact structure on $\SSS^3$, whose page $\Sigma$ is non-sporadic and has homology $H_1(\Sigma)$ of rank at least $2$; for example, one can take repeated positive stabilizations of the trivial open book.

Notice that, topologically, the Bourgeois manifold associated to these open books is of the form $\SSS^{2n-1} \times \TT^2$. 
We then perform subcritical/flexible surgeries to topologically obtain a sphere $\mathbb S^{2n+1}$. 
More precisely, first observe that, by the $h$-principle for isotropic embeddings \cite[Section 12.4]{EliMisBook} and for loose Legendrians \cite{MurLoose} (see also \cite[Theorem 7.16]{CieEliBook}), one can always realize smooth surgeries of index $1$ and $2$ via contact surgeries. 
In view of this, we first perform two index $1$ surgeries that kill the generators of the fundamental group in the torus factor, to obtain a simply connected manifold with the homology of $\mathbb{S}^{2n-1}\times \mathbb{S}^2$. 
Notice that, since these generators are non-trivial in homology and we start from a Bourgeois contact structure with trivial $c_1$, the resulting cobordism will have trivial $c_1$. 

Now, observe that the generator of the second homology group of the obtained manifold is spherical by the Hurewicz theorem. 
Moreover, using the fact that the ambient dimension is at least $5$, we can perturb a continuous map $f\colon \SSS^2\to M$ representing such a generator to a smooth embedding $\widetilde{f}\colon \SSS^2\to M$ according to standard transversality theory (see e.g.\ \cite[Theorems 2.6 and 2.13]{HirBook}). 
We can then perform an index $2$ surgery along $\widetilde{f}(\SSS^2)$, which has the effect of killing the generator in degree-$2$ homology.

In order to ensure that this surgery is possible, one needs to make sure that the normal bundle is trivial; however, this is ensured by the assumption that $c_1$ vanishes, since the obstruction to the normal bundle being trivial is given by the second Stiefel-Whitney class $w_2$, to which $c_1$ reduces modulo $2$. 
Finally, the corresponding handle attachments applied to the boundary of the null-cobordism $\DD^{2n} \times \mathbb \TT^2$ yield a simply connected homology ball $X$, which is then smoothly a ball by the $h$-cobordism theorem, giving that the manifold obtained by these surgeries is indeed a smoothly standard sphere. 
Realizing these surgeries as Weinstein handle attachment according to the above-mentioned $h$-principles, we thus obtain the desired contact structure $\xi_{ex}$ on $\SSS^{2n+1}$.

\smallskip

\textbf{Step 2: algebraic tightness.} Now, we claim that $(\SSS^{2n+1},\xi_{ex})$ is \emph{algebraically tight}. 
Indeed, the Bourgeois structures before surgeries are algebraically tight: for $n=2$ this is the content of \Cref{cor:alg_tight_Bour}, while for $n\geq 3$ these structures are in fact $1$-ADC by \Cref{ex:Brieskorn_ADC}, hence algebraically tight in view of \Cref{prop:property}. 
Then, since algebraic tightness is preserved under contact surgery, the claim follows.

\smallskip

\textbf{Step 3: the almost contact class.} As explained in Step 1, there is a Weinstein cobordism from a Bourgeois structure on $\SSS^{2n-1}\times \TT^2$, associated to the standard contact structure on $\SSS^{2n-1}$ (for some open book that is here irrelevant), to the contact sphere $(\SSS^{2n+1},\xi_{ex})$.
According to \Cref{rem:almost_cont}, this is in particular an \emph{almost complex} cobordism from a split almost contact structure $\eta:=(\xi_{std}, d\alpha_{std})\oplus (T\TT^2,\Omega)$ to the underlying almost contact structure $\zeta$ to $\xi_{ex}$ on $\SSS^{2n+1}$.
Now, as $\eta$ admits as almost complex filling $(B^{2n}\times \TT^2,\omega_{std}\oplus \Omega)$, this gives an almost complex filling $(X,J)$ of $(\SSS^{2n+1},\xi_{ex})$,
where $X$ is exactly the simply connected homology ball, and hence smoothly a ball by the h-cobordism theorem, that was described in Step 1. 
In particular, as there is only one almost complex structure up to homotopy on the ball, the almost contact structure $\zeta$ underlying $\xi_{ex}$ is homotopic to the standard one on $\SSS^{2n+1}$, as desired.

\smallskip

\textbf{Step 4: non-fillability.} To obstruct fillability, we want to apply \Cref{thm:unique} to the filling given by stacking the hypothetical filling of the contact structure on the sphere with the strong cobordism of \Cref{prop:cap_gen}, and so we first check the hypotheses of this proposition.
First, we remark that the $1$-surgeries can all be done along curves lying in a small neighborhood of a fixed $\TT^2$-fiber, since the isotropic $h$-principle is $C^0$-small, and that the $2$-surgery can be performed on a sphere contained in the trace of this small neighborhood under the initial surgery. 
In particular, all can be done on one side of the $\SSS^1$-equivariant decomposition for the Bourgeois contact structure.
Now, we have $V_+ = \Sigma \times D^*\SSS^1$ and $M'_+$ is the resulting flexibly fillable contact manifold given by first capping then doing surgery as in \Cref{prop:cap_gen}. The flexible filling $W_0$ of $M'_+$ is obtained from attaching two $2$-handles and one flexible $3$-handle along a loose Legendrian sphere to the subcritical domain $V_+\times \DD^2$.
If $(\SSS^{2n+1},\xi_{ex})$ has a strong filling $W$, we get a strong filling $W'$ of $M'_+$ by stacking the strong cobordism of \Cref{prop:cap_gen} on $W$.
Note that there exists a homology class $\beta \in H_n(M'_+;\Q)$ that is non-trivial in the flexible filling $W_0$ and is contained in  $V_+\times \{pt\}$. 
Indeed, since we can always find a non-trivial homology class $\beta'\in H_n(W_0;\Q)$ contained in $V_+\times \{pt\}$ by our assumption on $\Sigma$ (recall that we arranged, at the beginning of the proof, that $H_1(\Sigma)$ has rank at least $2$ when $n=2$), we can take $\beta$ to be any preimage of $\beta'$ under the surjective map $H_n(M'_+;\Q)\to H_n(W_0;\Q)$. 
Since $V_+\times \{pt\}$ can be isotoped to $\SSS^{2n+1}$ in the cobordism, $\beta$ must be trivial in the glued strong filling contradicting \Cref{thm:unique} for $M=\emptyset$.
\end{proof}

\begin{remark}[Freedom of choice of Open Book]
For concreteness, we chose to take the open book determined by the $A_1$-singularity in the proof above. 
This said, any nontrivial open book whose binding is $(-1)$-index positive and such that the page has a non-trivial middle homology group would suffice. 
\end{remark}

\begin{theorem}\label{thm:kill_strong}
	Let $(\SSS^{2n+1},\xi_{ex})$ be an exotic contact sphere in \Cref{thm:tightnonfill_alge}.
    Then $(M,\xi)\#(\SSS^{2n+1},\xi_{ex})$ is not strongly fillable for any contact manifold $M$ if $n\ge 3$. When $n=2$, the same holds if $c_1(M)$ vanishes.
\end{theorem}
\begin{proof}
	Note that from the proof of \Cref{thm:tightnonfill_alge} (Step 4), there exists a strong cobordism from $(\SSS^{2n+1},\xi_{ex})$ to a flexibly fillable contact manifold  $M'_+$ with vanishing first Chern class, such that there is a non-trivial homology class $\beta \in H_{n}(M'_+;\Q)$ annihilated in the cobordism but not in the flexible filling. 
	As the cobordism is obtained by blowing-down a Giroux domain and contact surgeries, which are operations performed away from a point, we get a strong cobordism from $(\SSS^{2n+1},\xi_{ex})\# M$ to $M'_+\# M$ and the homology class corresponding to $\beta$ is again annihilated in the cobordism. 
	We then apply \Cref{thm:unique} to obstruct fillability by stacking the blow-up cobordism on top of any hypothetical filling of $(\SSS^{2n+1},\xi_{ex})\# M$.
    (Note that the additional assumption of $c_1(M)=0$ in the case of $n=2$ is needed in order to apply \Cref{thm:unique}.(a) in this non-subcritical case.)
\end{proof}

\begin{proof}[Proof of \Cref{thm:tightnonfill_general}]
We can assume $(M,\xi)$ admits a strong filling, for otherwise, we have nothing to prove. 
In this case the connected sum $(M,\xi)\#(\SSS^{2n+1},\xi_{ex})$  is not strongly fillable by \Cref{thm:kill_strong}. 
We claim that algebraic tightness follows from the properties listed in \Cref{sec:alg_tight}.

Indeed, by \Cref{prop:property} (4), we know that $(M,\xi)$ is algebraically tight.
Then, by \Cref{prop:property} (2) and \Cref{thm:tightnonfill_alge}, the disjoint union $(M,\xi)\sqcup (\SSS^{2n+1},\xi_{ex})$ is algebraically tight as well. 
Since there is a Weinstein cobordism from  $(M,\xi)\sqcup (\SSS^{2n+1},\xi_{ex})$ to the connected sum $(M,\xi)\#(\SSS^{2n+1},\xi_{ex})$ by a Weinstein $1$-handle,   $(M,\xi)\#(\SSS^{2n+1},\xi_{ex})$ is algebraically tight by \Cref{prop:property} (3). 
\end{proof}

\subsection{Abundance of exotic structures on spheres}
Let $(\SSS^{2n+1},\xi_{ex})$ denote a tight non-fillable contact structure constructed in \Cref{thm:tightnonfill_alge}. This was obtained (for dimension at least $7$) via subcritical surgery on the Bourgeois manifold $BO(D^*\SSS^{n-1},\tau)$, which, as originally observed by Lisi, has an associated \emph{spinal open book decomposition}; c.f.\ \cite[Section 2]{BGM} for details. 
Namely, we have a decomposition
$$
BO(D^*\SSS^{n-1},\tau)=S^*\SSS^{n-1}\times D^*\mathbb
T^2\bigcup_{S^*\SSS^{n-1}\times \TT^3} \mbox{Map}(D^*\SSS^{n-1},\tau)\times \mathbb T^2,
$$
where $\mbox{Map}(D^*\SSS^{n-1},\tau)$ denotes the mapping torus of the Dehn-Seidel twist $\tau: D^*\SSS^{n-1} \to D^*\SSS^{n-1}$.

\begin{lemma}\label{lemma:sphere_ADC}
	When $n\ge 3$, the contact manifold $(\SSS^{2n+1},\xi_{ex})$ is $(n-3)$-ADC.
\end{lemma}
\begin{proof}
	\Cref{lem:kADC_stable} gives the indices of contractible Reeb orbits on the Bourgeois manifold $BO(D^*\SSS^{n-1},\tau)$. 
	On the other hand, the subcritical surgeries kill part of the topology of $BO(D^*\SSS^{n-1},\tau)$, making non-contractible orbits of $BO(D^*\SSS^{n-1},\tau)$ contractible in $(\SSS^{2n+1},\xi_{ex})$.
	We then need to compute the Conley-Zehnder indices of them as well.
	
	For simplicity we pick a global trivialisation of the almost contact structure underlying the Bourgeois contact structure, which is split as an almost contact structure according to \Cref{rem:almost_cont}, by simply taking a trivialization of $(\SSS^{2n-1},\xi_{st})$ and trivializing the tangent bundles of the $\TT^2$-fibers in an equivariant manner. 
	The Reeb orbits on $BO(D^*\SSS^{n-1},\tau)$ then have Conley-Zehnder indices as follows: 
	\begin{enumerate}
		\item On $B\times D^*\TT^2$, where $B=S^*\SSS^{n-1}$, the contact form is Morse-Bott. The computation of the Conley-Zehnder indices is in \Cref{lem:kADC_stable}. Namely, we have
		$$
		\mu_{LCZ}(\gamma_q)=\mu_{LCZ}(\gamma_B). 
		$$
		Here, $\gamma_q=\gamma_B\times \{(r=0,q)\}$ is the contractible orbit on the binding. And $\gamma_B$ is the Reeb orbits on $B$, which comes in a Morse-Bott family.
		From the computations in \cite[p.\ 37]{KvK16} for the $A_1$-singularity
		, we obtain that $\mu_{LCZ}(\gamma_B)\geq n-2$, and thus $$\mu_{LCZ}(\gamma_q)\ge n-2,$$ for every orbit in this region.
		
		\item On the product $\mbox{Map}(D^*\SSS^{n-1},\tau)\times \TT^2$, we have the following situation.
        Recall from \Cref{thm:bourgeois}, that the Bourgeois form $\alpha_{BO}$ in this region simply can be written as $\pi^*\alpha_{T} + \lambda_{fib}$ on the total space of the natural projection $\pi\colon \mathrm{Map}(D^*\SSS^{n-1},\tau)\times \TT^2\to \SSS^1\times \TT^2=\TT^3$, where $\alpha_T=Cd\theta + \cos\theta \, d x - \sin\theta\, d y$ on $\TT^3\ni (\theta,x,y)$ and $\lambda_{fib}$ is a $1$-form that is Liouville on each fiber of $\pi$.
        Moreover, it follows from Equation \eqref{eqn:reeb_vf} that in this region the Reeb vector field $R$ is just a lift via $d\pi$ of the Reeb vector field $R_T$ of $\alpha_T$ on $\TT^3$.
        More precisely, if $(p,x,y)\in \mathrm{Map}(D^*\SSS^{n-1},\tau)\times\TT^2$ lies in the fiber over $(\theta,x,y)\in \TT^3$, the Reeb orbit through it is tangent to $\{p\}\times \TT^2$ and is a geodesic for the flat metric, that has starting vector at angle (w.r.t.\ the coordinates $(x,y)$) exactly $\theta$.
        In other words, Reeb orbits are naturally organized in $D^*\SSS^{n-1}\times \mathbb S^1$-Morse--Bott families, each of which projects down via $\pi$ to a $\SSS^1$-Morse--Bott family of orbits of $R_T$ on $\TT^3$. 
        
        Now, perturbing with a Morse function on $\SSS^1$ with critical points $p_1,p_2$, one can find another contact form $\alpha_T'$ that is $C^\infty$-close to $\alpha_T$ on $\TT^3$ such that each $\SSS^1$-Morse--Bott family of Reeb orbits corresponds to a pair of Reeb orbits. 
        We denote by $\gamma_{p_j}$ the Reeb orbit on $\mathbb T^3$ corresponding to $p_j$.
        By standard properties of the  Conley-Zehnder index, the minimal Conley-Zehnder index of the orbits on $\mathbb T^3$ after such perturbation is zero.
		
        We can then look at the perturbation $\alpha_{BO}'$ given by $\pi^*\alpha_T'+\lambda_{fib}$.
        Likewise, we can further locally perturb the contact form by a Morse function with a unique local maximum on $D^*\mathbb S^{n-1}$ and two critical points $q_1,q_2$ lying in the zero section, respectively of index $2n-2$ and $n-1$. 
        We denote by $\gamma_{q_i,p_j}$ the non-degenerate orbit corresponding to the pair $(q_i,p_j)$.
		Again, by standard properties of the Conley-Zehnder index, we then have:
		\begin{equation}\nonumber
		\begin{split}
		\mu_{CZ}(\gamma_{q_i,p_j})=&\;\mbox{ind}(q_i)-\frac{1}{2}\dim (D^*\mathbb S^{n-1})+\mu_{CZ}(\gamma_{p_j})\\
		=&\; \mbox{ind}(q_i)-n+1+\mu_{CZ}(\gamma_{p_j})\\
		\geq&\; n-1 - n+1+0=0.
		\end{split}
		\end{equation}
		Therefore, the Conley-Zehnder index of every orbit is non-negative along this region $\mathrm{Map}(D^*\SSS^{n-1},\tau)\times \TT^2$, and their SFT degrees are at least $n-2$.

		\item Along the interface region $S^*\SSS^{n-1}\times \TT^3$, the situation is modeled as a smoothened corner of the product $D^*\SSS^{n-1}\times D^*\TT^2$.
		Computations in this setting are carried out in detail in \cite[Proposition 6.18 (3)]{Zho21}, and the result is that
		the minimal Conley-Zehnder index is the sum of minimal Conley-Zehnder indices on $S^*\SSS^{n-1}$ (i.e.\ $n-2$) and $S^*\TT^2$ (i.e.\ $0$). 
        (Note that \cite{Zho21} considered the $S^1$-family of Reeb orbits, whose generalized Conley-Zehnder index has an extra $1/2$. 
        On the other hand, here we consider non-degenerate Reeb orbits after a small perturbation of the Morse-Bott family, hence the minimal Conley-Zehnder index adds.)
		Hence, we have $\mu_{CZ}(\gamma)\ge n-2$ for every orbit $\gamma$ in this region. 
		
		\item\label{4} The attaching of two $2$-handles and a $3$-handles
        create contractible new Reeb orbits, but these have Conley-Zehnder indices at least $n-1$ asymptotically, by the proof of \cite[Theorem 3.14]{Laz20} (or \cite[Proposition 2.3]{zbMATH07706508} for more precise numbers). 
		When we attach the two $2$-handles to kill the fundamental group of $BO(D^*\SSS^{n-1},\tau)$.
        Here, to maintain the correctness of Conley-Zehnder indices computed above, we need to choose the framing of the symplectic normal bundle of the isotropic circles so that the induced trivialization of the contact structure (used in cases 2 and 3 above) along the isotropic circle extends to a trivialization of the core disk in the handle. 
		In fact, since the $\TT^2$-fibers over the binding of the open book in the base are isotropic, these give natural choices of isotropic circles for surgery (of any slope). 
		One then simply takes the framing of the normal bundle induced from the trivialisation of the symplectic normal bundle induced from the projection to the base.
	\end{enumerate}
    In total, we have shown that the Bourgeois contact form considered in \S \ref{sec:Bourgeois} is $(n-3)$-index positive even including all the non-contractible Reeb orbits with a suitable framing (called $\pi_1$-sensitive ADC in \S \ref{sec:exact_non-Weinstein}). In particular, it is $(n-3)$-ADC after a small perturbation as in \cite[Lemma 2.3]{Bothesis}. The framing on the Bourgeois contact structure is compatible with framing on $\mathbb{S}^{2n+1}$ after the surgery as explained in point \eqref{4}, the subcritical surgery will preserve the $(n-3)$-ADC property by \cite[Theorems 3.14]{Laz20}\footnote{As the proof of \cite[Theorem 3.14]{Laz20} shows that index $1$ or $2$ subcritical handle to a $(2n+1)$-dimensional contact manifold adds Reeb orbits of Conley-Zehnder index at least $n$ asymptotically. It would preserve $p$-ADC as long as $p<2n-2$.}, or \Cref{prop:pi_1ADC} below. 
\end{proof}

As proved in \cite{Laz20}, the ADC condition is sufficient to obtain well-defined invariants of Liouville-fillable contact manifolds using positive symplectic cohomology (generated by contractible orbits) of any of its \emph{topologically simple} Liouville fillings (that is, Liouville fillings with $c_1=0$ and so that the inclusion of the boundary induces an injective map on $\pi_1$).
Under a slightly stronger assumption, namely that there is a non-degenerate Reeb vector field all of whose contractible orbits have SFT degree at least $2$,  Cieliebak-Oancea \cite[Section 9.5]{CieOan18} described a version of symplectic homology on the trivial symplectic cobordism $M\times [0,1]$ associated to the contact manifold $M$. 
The latter construction yields a contact invariant in view of a neck-stretching argument.
In fact, for exactly the same reason as in the fillable case \cite{Laz20}, it is enough to have $1$-ADC for this to be a well-defined invariant of the contact manifold itself. With a slight abuse of notation, we will then denote the positive symplectic cohomology of a $1$-ADC contact manifold $(M,\xi)$ simply by $SH_+^*(M)$ (i.e.\ $SH^{>0}_*(M)$ in \cite[Section 9.5]{CieOan18}). When $M$ is the contact boundary of a topologically simple Liouville filling $W$, we have  $SH_+^*(M)=SH_+^*(W)$.

\begin{proposition}\label{prop:brieskorn}
When $n\ge 5$, there exists a Weinstein fillable contact structure  $(\SSS^{2n+1},\xi')$ homotopic to the standard almost contact structure, which is $1$-ADC, and such that $0<\dim SH^{2n-3}_+(\SSS^{2n+1},\xi')<\infty$.
\end{proposition}

The construction of $\xi'$ in \Cref{prop:brieskorn} is based on the description given in \cite[\S 5.3.1]{KvK16} of the Reeb orbits of an explicit contact form on the Brieskorn manifold $\Sigma(a_0,\ldots,a_n):= \{z_0^{a_0}+\ldots+z_n^{a_n}=1\}\cap S^{2n+1}$. 
We first set up some terminology that we will need later on in the description of the Reeb flow from \cite[\S 5.3.1]{KvK16}.

A positive real number $T$ is called a \textbf{principal period} if there exists $I_T\subset I_n=\{0,\ldots,n\}$ such that:

\begin{enumerate}
    \item $a_i|T$ if and only if $i\in I_T$; 
\item $T=\lcm \{a_i\mid i \in I_T\}$; 
\item $|I_T|\ge 2$.
\end{enumerate}
Then, according to \cite[\S 5.3.1]{KvK16}, the periods of the Reeb flow are of the form $NT$, for some $N\in \NN_+$ and a principal period $T$. 
For a period $A$, there is a principal period $T$ such that $A=NT$ and there is no other principal period $T'$ that divides $NT$ and is divided by $T$. 
There is then a family of dimension $2|I_T|-3$ (diffeomorphic to the Brieskorn manifold with coefficients given by the sequence $I_T$) of parametrized Reeb orbits with period $NT$, which, after taking the quotient by the action given by $S^1$-reparameterization becomes $2|I_T|-4$ dimensional.
The generalized Conley-Zehnder index of such a family is given by 
\begin{equation*}
2\sum_{i\in I_T}\frac{NT}{a_i}+2\sum_{i\notin I_T}\left\lfloor\frac{NT}{a_i}\right\rfloor+|I_n|-|I_T|-2NT.
\end{equation*}
If we perturb the contact form into a non-degenerate one using a Morse function on the critical manifold following \cite[Lemma 2.3]{Bothesis}, the minimal Conley-Zehnder index of the non-degenerate orbits coming from this family is given by the generalized Conley-Zehnder index minus half of the dimension of the family itself, i.e.\
\begin{equation}\label{eqn:CZ}
2\sum_{i\in I_T}\frac{NT}{a_i}+2\sum_{i\notin I_T}\left\lfloor\frac{NT}{a_i}\right\rfloor+|I_n|-2|I_T|-2NT+2.
\end{equation}

Now, for each $A$, there might be several such principal periods $T$. 
Each will give rise to a Morse-Bott non-degenerate family of Reeb orbits, and they describe all of the Reeb orbits.

\begin{proof}[Proof of \Cref{prop:brieskorn}]
	We consider the Brieskorn sphere $Y:=\Sigma(n,\ldots,n, n+1, p)$, where $p$ is a prime number with $p\gg n$. 
	The minimal  Conley-Zehnder index of a small perturbation of the standard contact form on  $Y$ is $4-n$, obtained from the expression in \eqref{eqn:CZ} above as $|I_{n+1}|=n+2$, $|I_T|=n$, $N=1$, $T=n$ in our setup. 
	The SFT degrees are then at least $2$. The minimal Conley-Zehnder indices of other Morse-Bott families of Reeb orbits of the natural contact form (after suitable perturbation) are $6-n$ (given by $N=2$, $T=n$ in \eqref{eqn:CZ}) as $p\gg n$, and so they have SFT degree $4$. 
	As a consequence, we have that $Y$ is $1$-ADC. 
	Moreover, by the Morse-Bott spectral sequence (see e.g.\ \cite[Theorem 5.4.]{KvK16})\footnote{Note that \cite{KvK16} uses homological grading by $\mu_{CZ}$ while we use cohomological grading $n+1-\mu_{CZ}$} and the index gap above, we deduce that 
    \[
    \dim SH_+^{(n+1)-(4-n)}(Y)=\dim SH_+^{2n-3}(Y)=1 \, ,
    \]
    which is generated by the family of minimal  Conley-Zehnder index. 
	Now, by \cite[Proposition 3.6]{KvK16}, $Y$ is homeomorphic to a sphere. 
	Then there exists $K$, such that the $K$-th iterated self-connected sum $\#^K Y$ is {\em diffeomorphic} to the standard sphere, as the group of smooth structures on $S^k$ is finite for $k\ne 4$, see \cite{zbMATH03188186}. 
	Since $\dim Y \ge 11$, the contact sum preserves the $1$-ADC property: indeed, by the proof of \cite[Theorem 3.14, 3.15]{Laz20}, an index $k$ subcritical handle attachment to a $2n-1$ dimensional contact manifold preserves the $p$-ADC property for any $p\le 2n-3-k$. 
    Moreover, by \cite[Theorem 7.1 and Proposition 9.19]{CieOan18} we have 
    \[
    \dim SH_+^{2n-3}(\#^K_{i=1} Y)=\dim (\bigoplus^K_{i=1} SH^{2n-3}_+(Y))=K \, .
    \]
	To obtain the standard almost contact structure, we choose $Y'$ to be the flexibly fillable sphere with the opposite homotopy class \footnote{To be precise, here ``opposite'' is to be interpreted as the inverse w.r.t.\ the natural group structure on the space of almost contact structures on the sphere given by the connected sum operation.} of almost contact structures as that of $\#^K Y$, as in \cite[Corollary 5.19]{zbMATH07567794}. 
	Then $Y'\#^K Y$ is a standard sphere with the standard almost contact structure, and it is $1$-ADC.
	Moreover, $SH^{2n-3}_+(Y')=0$ when $n\ge 5$. 
	Hence, we may take $Y'\#^K Y$ as the desired contact sphere.
\end{proof}

We now have all the needed ingredients to prove  \Cref{thm:tightnonfill_general_infinite}.

\begin{proof}[Proof of \Cref{thm:tightnonfill_general_infinite}]
	One first realizes the given almost contact structure as the contact boundary $(M,\xi_{\textit{flex}})$ of a flexible Weinstein domain using Eliashberg's $h$-principle \cite{CieEliBook}. 
	For $i\in \NN$, we then consider the iterated connected sums 
    \[
    Y_i:=(\SSS^{2n+1},\xi_{ex}) \# (M,\xi_{\textit{flex}})\#^i(\SSS^{2n+1},\xi')
    \]
    with the contact sphere $(\SSS^{2n+1},\xi_{ex})$ from \Cref{thm:tightnonfill_alge} and $(\SSS^{2n+1},\xi')$ from \Cref{prop:brieskorn}. 
	Notice that $M$ is $1$-ADC, as it is flexibly fillable \cite{Laz20}, and that $SH^*_+(M)$ is supported in degree smaller than $n+1$, in view of the vanishing of symplectic homology for flexible Weinstein domains.
	Since $SH_+^*(\SSS^{2n+1},\xi_{ex})$ is supported in degrees at most $n+1$ by \Cref{lemma:sphere_ADC}, then if $2n-3>n+1$, i.e.\ $n\ge 5$, we have $\dim SH^{2n-3}_+(Y_i)=N\cdot i$, where $N=\dim SH^{2n-3}_+(\SSS^{2n+1},\xi')$,  by \cite[Theorem 7.1, Proposition 9.19]{CieOan18}. 
	As a consequence, the $Y_i$'s are pairwise distinguished by $SH^*_+$. 
    Moreover, they are tight and not strongly fillable by the same arguments as in the proof of \Cref{thm:tightnonfill_general}.
\end{proof}

\begin{remark}
When $n=3$, $(\SSS^{2n+1},\xi_{ex})$ is not $1$-ADC, which costs us the well-definedness of the positive symplectic cohomology. 
When $n=4$, it is not easy to find a degree such that positive symplectic cohomology is different for varying summands in the connected sum.
In those cases, we can use positive symplectic cohomology for augmentations in \cite[\S 3.2.3]{zhou2023contact}\footnote{The positive symplectic cohomology of $1$-ADC manifolds is the special case for the trivial augmentation.}, which was motivated by the work of Bourgeois-Oancea \cite{MR2471597}. 
Then we can obtain a ``contact invariant" by enumerating through all possible augmentations.
However, the current status of the contact homology \cite{pardon2019contact,MR4539062} is not enough to guarantee that such a contact invariant is independent of various choices. 
This being said, \emph{assuming the invariance of such a theory}, one can upgrade \Cref{thm:tightnonfill_general_infinite} to cover the cases of $n=3,4$. 
\end{remark}

%% file: s6.tex
\section{Liouville fillable contact spheres without Weinstein fillings}\label{sec:exact_non-Weinstein}
We now explain how to use surgery techniques to deduce the existence of Liouville fillable contact structures that are not Weinstein fillable on spheres. This requires a refinement of the surgery arguments used above, and we begin with the following:

\begin{lemma}\label{lem:surg_parallel}
Let $(M^{2n+1},\xi)$ be a contact manifold of dimension at least $5$ whose underlying almost contact structure is stably trivial (as an almost contact structure). Then there is an almost Weinstein cobordism $W$ from $(M,\xi)$ to a homotopy sphere $\Sigma$ (with some almost contact structure).
\end{lemma}
\begin{proof}
This follows immediately from results in \cite{BCS1}, building on the work of Kreck \cite{Kreck}. 
For the reader's convenience, we describe the steps in the proof. 

First, since the contact structure is stably parallelizable, it follows that the tangent bundle of $M$ is stably trivial (forgetting the complex structure) so that the classifying map of its tangent bundle factors through a point: 
\[
\xymatrix{
 & pt \ar[d] \\
M \ar[ur]^{\overline{\mathcal{T}}_M} \ar[r]_{\mathcal{T}_M} & BO.
} \]

By surgery as in \cite[Lemma 2]{Kreck}, we can add handles to $M \times I$ to inductively kill homotopy groups up to dimension $n$; let $W$ be the corresponding cobordism. 
Notice that the positive boundary of $W$ is a homotopy sphere by construction.

We now claim that the classifying map of the stable tangent bundle of $W$ is constant as well. 
This can be argued as follows.
First, the stable triviality of the tangent bundle of $M$ ensures that the normal bundles of the embedded spheres on which we do surgery are again stably trivial. 
Because the codimensions of the attaching spheres are here $n+1 \geq 3$, this is moreover equivalent to actual triviality. 
Lastly, one can choose framings in such a way that the stable trivialisation of the tangent bundle extends over the handles to all of $W$ and the claim then follows.

Thus, since $TW$ is stably trivial as a real bundle, we can endow a stabilization of it with an almost complex structure, i.e.\ we found a stable almost complex structure on $W$.
The conclusion of the lemma then directly follows from the fact that, according to \cite[Lemma 2.16]{BCS1}, one can destabilize the stable almost complex structure on $W$ that we just found to a genuine almost complex structure which agrees (up to homotopy) with the given one on the negative end. Since this cobordism was obtained by attaching handles of index at most $n$ it is, in particular, almost Weinstein.
\end{proof}

As we have already seen in \Cref{lemma:sphere_ADC}, it is important to keep track of the Conley-Zehnder indices for non-contractible orbits to understand the indices of the Reeb orbits that become contractible on the spheres obtained via the flexible surgeries. 
For this, assume the contact manifold $Y$ has vanishing rational first Chern class.  
If we fix a complex trivialisation $\tau$ of the contact structure (or of the sum of the right number of copies of the contact structure, if the first Chern class is a non-zero torsion) over the $2$-skeleton, or equivalently a trivialisation of the complex determinant bundle associated to the contact structure, then we can speak of the Conley-Zehnder indices for any orbits with respect to $\tau$, which we will denote by $\mu_{CZ}^\tau(\gamma)$. 
One can then define notions of index positivity and ADC that take into account all periodic orbits and not just those that are contractible. This leads to the following concept of $\pi_1$-sensitive ADC manifolds introduced in \cite[Definition 3.5]{intersection}.

\begin{definition}\label{def:ADC-pi_1}
Let $(Y,\xi)$ be a contact manifold such that $c_1^{\Q}(\xi)=0$. Let $\tau$ be a trivialization of $\det_{\CC}\left(\oplus^N\xi\right)$ for some $N\in \NN_+$. We say $(Y,\xi,\tau)$ is \emph{$\pi_1$-sensitive ADC} if there exist contact forms $\alpha_1>\alpha_2>\ldots$, and positive real numbers $D_1<D_2<\ldots$ converging to infinity, such that \emph{all} Reeb orbits of $\alpha_i$ of period up to $D_i$ are non-degenerate and have rational SFT grading $\mu_{CZ}^{\tau}(\gamma)+n-3>0$.
\end{definition}

\begin{proposition}
\label{prop:product_ADC}
Suppose that $V$ is Liouville with $c_1(V)$ torsion, then the contact boundary $\partial(V\times \DD^2)$ is $\pi_1$-sensitive ADC with respect to any trivialization $\tau$.
\end{proposition}
\begin{proof}
This proposition was essentially established in \cite[Theorem 6.3]{Zho21}; we give some further details here for completeness. 

We first choose a trivialisation $\tau'_0$ of $\det_{\CC} \oplus^N TV$ for some $N\in \NN_+$ over $V$, which induces a trivialisation $\tau_0$ of $\det_{\CC} \oplus^N T(V\times \DD^2)$ over $V\times \DD^2$. 
Following the proof of \cite[Theorem 6.3]{Zho21} and the notations therein, $\partial(V\times \DD^2)$ can then be decomposed into $Y^2_{\rho, f,g}$, diffeomorphic to $V\times S^1$, and $Y^1_{\rho, f,g}$, diffeomorphic to $\partial V \times \DD^2$. We obtain a sequence of decreasing contact forms by choosing $\rho,f,g$. Here $f,g$ are perturbation functions, and $\rho$ is a positive number measuring the size of $V$ compared to $\DD^2$, i.e.\ a large $\rho$ results in a contact form that is ``thin in the $\DD^2$ direction'' (such as what happens for the contact form at the boundary of a thin ellipsoid), see \cite[Proof of Theorem 6.3]{Zho21} for details. 
The Reeb orbits (with period up to some threshold) on $Y^2_{\rho,f,g}$ are contractible and have positive SFT degrees. Since those orbits are contractible, the SFT degree is independent of the trivialization \cite[Proposition 3.8]{gironella2021exact}. The Reeb orbits on $Y^1_{\rho,f,g}$ are described in \cite[Proposition 6.7]{Zho21}, where the Conley-Zehnder indices w.r.t.\ $\tau_0$ are computed using the Conley-Zehnder indices of the $\partial V$-component w.r.t.\ $\tau'_0$. In particular, in this sequence of contact forms, those orbits have arbitrarily high Conley-Zehnder indices w.r.t.\ $\tau_0$. Heuristically speaking, those orbits are similar to the orbits rotating along the long axis of a very thin ellipsoid. 

Lastly, for any trivialization $\tau$ of $\det_{\CC} \oplus^N T(V\times \DD^2)|_{\partial(V\times \DD^2)}$, the difference from the Conley-Zehnder index of $\gamma$ w.r.t.\ $\tau$ and $\tau_0$ is given by $\frac{2}{N}\la \tau-\tau_0,\gamma\ra$, where we view $\tau-\tau_0$ as a class in $H^1(\partial(V\times \DD^2))$, see \cite[\S 3.3.2]{gironella2021exact}. Therefore, the error is a bounded number for any fixed homology class of $\gamma$. On the other hand, for any fixed homology class,  Reeb orbits on $Y^1_{\rho,f,g}$ in that class have arbitrarily high Conley-Zehnder indices w.r.t.\ $\tau_0$ from the proof of \cite[Theorem 6.3]{Zho21},  this implies that $\partial(V\times \DD^2)$ is $\pi_1$-sensitive ADC w.r.t.\ $\tau$.
\end{proof}

Lazarev \cite[Theorems 3.14, 3.15]{Laz20} proved that the ADC property is preserved under subcritical and flexible surgeries as long as the vanishing of the rational first Chern class condition holds. The case of attaching a $2$-handle \cite[Theorem 3.14]{Laz20} has extra assumptions since Lazarev only considered contractible Reeb orbits. In the case of $\pi_1$-sensitive ADC manifolds,  \cite[Theorems 3.14, 3.15]{Laz20} holds with identical proofs as long as the trivializations are compatible: 

\begin{proposition}\label{prop:pi_1ADC}
    Let $W$ be a flexible cobordism from $Y_-$ to $Y_+$ with $c_1^{\Q}(TW)=0$. Let $\tau$ be a trivialization of $\det_{\CC}\oplus^N TW$ whose restriction to $Y_{\pm}$ is denoted by $\tau_{\pm}$. If $Y_-$ is $\pi_1$-sensitive ADC w.r.t.\ $\tau_-$, $Y_+$ is $\pi_1$-sensitive ADC w.r.t.\ $\tau_+$.
\end{proposition}

We now prove \Cref{thm:exact_non_Stein_sphere} from the Introduction. 

\begin{proof}[Proof of \Cref{thm:exact_non_Stein_sphere}]
We let $V$ be a Liouville domain as constructed by Massot, Niederkr\"uger, and Wendl \cite{MNW}, so that $V \cong N \times I$ and $N^{2n-1} = G / \Gamma$ is a quotient of a solvable Lie group by a co-compact lattice. 
More precisely, the contact structures on the boundary components of $V$ come from left-invariant structures on $G$, and are hence parallelizable. 
The same holds for the (symplectic) tangent bundle on $V$, and thus for $V \times \DD^2$.

Now set $M = \partial (V \times \DD^2)$ to be the contact boundary of the product. The underlying almost contact structure is stably trivial, and we can apply \Cref{lem:surg_parallel} to obtain an almost Weinstein cobordism $W$ to some homotopy sphere. Since the almost complex structure on $W$ is obtained from destabilizing a trivial complex structure, the first Chern class of $W$ is zero. Combining \Cref{prop:product_ADC} with \Cref{prop:pi_1ADC}, and realizing $W$ as a flexible Weinstein cobordism using \cite[Theorem 13.1]{CieEliBook}, we obtain an ADC contact structure on the convex end, which is a homotopy sphere $(\Sigma,\xi_{\Sigma})$. 
Moreover, this has a natural Liouville filling obtained by attaching Weinstein handles to $V \times \DD^2$, which in particular has trivial symplectic cohomology by \cite[Theorem 1.11]{Cie02b} and \cite[Theorem 5.6]{ BourEkEl}.
The cohomology of the resulting filling has a non-trivial element in degree $2n-1$, which evaluates non-trivially on the fundamental class of $N$. 

Using the ADC property and the vanishing of its symplectic cohomology, according to \cite[Corollary B]{Zho21} the cohomology of any Weinstein filling of $(\Sigma,\xi_{\Sigma})$ has cohomology identical to the cohomology of the natural Liouville filling, and hence has non-trivial cohomology in degree $2n-1$, which is bigger than $n+1$  under our assumption of $n \ge 3$.  
So we deduce that $(\Sigma,\xi_{\Sigma})$ is Liouville, but not Weinstein, fillable. Such phenomena were observed in \cite{Zho21}.

Now since the set of smooth, oriented homotopy spheres with the operation of connected sum forms a finite group, up to taking self-connected sums we can assume that $\Sigma \cong \mathbb S^{2n+1}$ is diffeomorphic to the standard sphere, which then also has an ADC contact structure with a Liouville filling with vanishing symplectic cohomology that is homologically not Weinstein. 
In other words, we have found a contact structure on the standard sphere which is Liouville but not Weinstein fillable. 

To get the standard almost contact structure, we use the fact that all stably trivial almost contact structures on a standard sphere can be realised via contact structures (cf.\ \cite[Lemma 2.17]{BCS1}) which are flexibly fillable. Taking a connected sum of one of such contact structures with the previously constructed Liouville non-Weinstein fillable structures then gives one structure as in the statement of \Cref{thm:exact_non_Stein_sphere}.
\end{proof}

One can now easily get infinitely many distinct contact structures as in \Cref{thm:exact_non_Stein_sphere} by taking the connected sum with the infinitely many examples of homotopically standard and flexibly fillable contact structures on spheres from \cite{Laz20}. Arguing {\em mutatis mutandis}, we also obtain \Cref{thm:exact_non_Stein_general} from the Introduction.
We also point out that the $5$-dimensional case remains open since, for all known examples of four-dimensional Liouville domains $V$ with disconnected boundary, the manifold $V \times \DD^2$ is Weinstein by Breen-Christian \cite{BC_arXiv}.

%% file: main.bbl
\newcommand{\etalchar}[1]{$^{#1}$}
\begin{thebibliography}{{Cie}02b}

\bibitem[Arn95]{Arn95}
V.~I. Arnol\'d.
\newblock Some remarks on symplectic monodromy of {M}ilnor fibrations.
\newblock In {\em The {F}loer memorial volume}, volume 133 of {\em Progr.
  Math.}, pages 99--103. Birkh\"auser, Basel, 1995.

\bibitem[Avd23]{avdek2020combinatorial}
Russell Avdek.
\newblock Combinatorial {Reeb} dynamics on punctured contact 3-manifolds.
\newblock {\em Geom. Topol.}, 27(3):953--1082, 2023.

\bibitem[BC25]{BC_arXiv}
Joseph Breen and Austin Christian.
\newblock Torus bundle liouville domains are stably weinstein.
\newblock {\em J. Topol.}, 18(4), 2025.

\bibitem[BCS14]{BCS1}
Jonathan Bowden, Diarmuid Crowley, and Andr\'{a}s~I. Stipsicz.
\newblock The topology of {S}tein fillable manifolds in high dimensions {I}.
\newblock {\em Proc. Lond. Math. Soc. (3)}, 109(6):1363--1401, 2014.

\bibitem[BCS15]{BCS2}
Jonathan Bowden, Diarmuid Crowley, and Andr{\'a}s~I. Stipsicz.
\newblock The topology of {Stein} fillable manifolds in high dimensions. {II}
  (with an appendix by {Bernd} {C}. {Kellner}).
\newblock {\em Geom. Topol.}, 19(5):2995--3030, 2015.

\bibitem[BEE12]{BourEkEl}
Fr\'{e}d\'{e}ric Bourgeois, Tobias Ekholm, and Yakov Eliashberg.
\newblock Effect of {L}egendrian surgery.
\newblock {\em Geom. Topol.}, 16(1):301--389, 2012.
\newblock With an appendix by Sheel Ganatra and Maksim Maydanskiy.

\bibitem[BEH{\etalchar{+}}03]{BEHWZ03}
Frédéric Bourgeois, Yakov Eliashberg, Helmut Hofer, Krzysztof Wysocki, and
  Eduard Zehnder.
\newblock Compactness results in symplectic field theory.
\newblock {\em Geom. Topol.}, 7:799--888, 2003.

\bibitem[BEM15]{BEM}
Matthew Borman, Yakov Eliashberg, and Emmy Murphy.
\newblock Existence and classification of overtwisted contact structures in all
  dimensions.
\newblock {\em Acta Math.}, 215(2):281--361, 2015.

\bibitem[BGM22]{BGM}
Jonathan Bowden, Fabio Gironella, and Agustin Moreno.
\newblock Bourgeois contact structures: {T}ightness, fillability and
  applications.
\newblock {\em Invent. Math.}, 230(2):713--765, 2022.

\bibitem[BGZ19]{BGZ}
Kilian Barth, Hansj\"{o}rg Geiges, and Kai Zehmisch.
\newblock The diffeomorphism type of symplectic fillings.
\newblock {\em J. Symplectic Geom.}, 17(4):929--971, 2019.

\bibitem[BH23]{MR4539062}
Erkao Bao and Ko~Honda.
\newblock Semi-global {K}uranishi charts and the definition of contact
  homology.
\newblock {\em Adv. Math.}, 414:Paper No. 108864, 148, 2023.

\bibitem[BN10]{bourgeois2010towards}
Fr\'{e}d\'{e}ric Bourgeois and Klaus Niederkr\"{u}ger.
\newblock Towards a good definition of algebraically overtwisted.
\newblock {\em Expo. Math.}, 28(1):85--100, 2010.

\bibitem[BO09a]{MR2471597}
Fr\'{e}d\'{e}ric Bourgeois and Alexandru Oancea.
\newblock An exact sequence for contact- and symplectic homology.
\newblock {\em Invent. Math.}, 175(3):611--680, 2009.

\bibitem[BO09b]{BODuke}
Fr{\'e}d{\'e}ric Bourgeois and Alexandru Oancea.
\newblock Symplectic homology, autonomous {Hamiltonians}, and {Morse}-{Bott}
  moduli spaces.
\newblock {\em Duke Math. J.}, 146(1):71--174, 2009.

\bibitem[Bou02a]{Bo}
Fr\'{e}d\'{e}ric Bourgeois.
\newblock Odd dimensional tori are contact manifolds.
\newblock {\em Int. Math. Res. Not.}, (30):1571--1574, 2002.

\bibitem[Bou02b]{Bothesis}
Frédéric Bourgeois.
\newblock {\em A {M}orse-{B}ott approach to contact homology}.
\newblock ProQuest LLC, Ann Arbor, MI, 2002.
\newblock Thesis (Ph.D.)--Stanford University.

\bibitem[Bow12]{Bow_exact}
Jonathan Bowden.
\newblock Exactly fillable contact structures without {S}tein fillings.
\newblock {\em Algebr. Geom. Topol.}, 12(3):1803--1810, 2012.

\bibitem[BvK10]{MR2646902}
Fr\'{e}d\'{e}ric Bourgeois and Otto van Koert.
\newblock Contact homology of left-handed stabilizations and plumbing of open
  books.
\newblock {\em Commun. Contemp. Math.}, 12(2):223--263, 2010.

\bibitem[CDvK16]{CDvK16}
River Chiang, Fan Ding, and Otto van Koert.
\newblock Non-fillable invariant contact structures on principal circle bundles
  and left-handed twists.
\newblock {\em Int. J. Math.}, 27(3):55, 2016.
\newblock Id/No 1650024.

\bibitem[CE12]{CieEliBook}
Kai Cieliebak and Yakov Eliashberg.
\newblock {\em From {Stein} to {Weinstein} and back. {Symplectic} geometry of
  affine complex manifolds}, volume~59 of {\em Colloq. Publ., Am. Math. Soc.}
\newblock Providence, RI: American Mathematical Society (AMS), 2012.

\bibitem[Cie02a]{Cie02b}
Kai Cieliebak.
\newblock Handle attaching in symplectic homology and the chord conjecture.
\newblock {\em J. Eur. Math. Soc. (JEMS)}, 4(2):115--142, 2002.

\bibitem[{Cie}02b]{Cie02}
Kai {Cieliebak}.
\newblock {Subcritical Stein manifolds are split}.
\newblock {\em ArXiv e-prints}, 2002.

\bibitem[CMP19]{CMP}
Roger Casals, Emmy Murphy, and Francisco Presas.
\newblock Geometric criteria for overtwistedness.
\newblock {\em J. Amer. Math. Soc.}, 32(2):563--604, 2019.

\bibitem[CO18]{CieOan18}
Kai Cieliebak and Alexandru Oancea.
\newblock Symplectic homology and the {E}ilenberg-{S}teenrod axioms.
\newblock {\em Algebr. Geom. Topol.}, 18(4):1953--2130, 2018.
\newblock Appendix written jointly with Peter Albers.

\bibitem[DG12]{DinGei12}
Fan Ding and Hansj\"{o}rg Geiges.
\newblock Contact structures on principal circle bundles.
\newblock {\em Bull. Lond. Math. Soc.}, 44(6):1189--1202, 2012.

\bibitem[DGZ14]{DGZ14}
Max D\"{o}rner, Hansj\"{o}rg Geiges, and Kai Zehmisch.
\newblock Open books and the {W}einstein conjecture.
\newblock {\em Q. J. Math.}, 65(3):869--885, 2014.

\bibitem[EH02]{EtnHon}
John~B. Etnyre and Ko~Honda.
\newblock Tight contact structures with no symplectic fillings.
\newblock {\em Invent. Math.}, 148(3):609--626, 2002.

\bibitem[Eli89]{Eli89}
Yakov Eliashberg.
\newblock Classification of overtwisted contact structures on {$3$}-manifolds.
\newblock {\em Invent. Math.}, 98(3):623--637, 1989.

\bibitem[Eli90]{Eli90}
Yakov Eliashberg.
\newblock Filling by holomorphic discs and its applications.
\newblock In {\em Geometry of low-dimensional manifolds, 2 ({D}urham, 1989)},
  volume 151 of {\em London Math. Soc. Lecture Note Ser.}, pages 45--67.
  Cambridge Univ. Press, Cambridge, 1990.

\bibitem[Eli92]{Eli92}
Yakov Eliashberg.
\newblock Contact {$3$}-manifolds twenty years since {J}. {M}artinet's work.
\newblock {\em Ann. Inst. Fourier (Grenoble)}, 42(1-2):165--192, 1992.

\bibitem[Eli96]{Eli96}
Yakov Eliashberg.
\newblock Unique holomorphically fillable contact structure on the {$3$}-torus.
\newblock {\em Internat. Math. Res. Notices}, (2):77--82, 1996.

\bibitem[EM02]{EliMisBook}
Yakov Eliashberg and Nikolai Mishachev.
\newblock {\em Introduction to the {$h$}-principle}, volume~48 of {\em Graduate
  Studies in Mathematics}.
\newblock American Mathematical Society, Providence, RI, 2002.

\bibitem[Fau20]{Fau20}
Alexander Fauck.
\newblock On manifolds with infinitely many fillable contact structures.
\newblock {\em Internat. J. Math.}, 31(13):2050108, 71, 2020.

\bibitem[Gay06]{Gay}
David~T. Gay.
\newblock Four-dimensional symplectic cobordisms containing three-handles.
\newblock {\em Geom. Topol.}, 10:1749--1759, 2006.

\bibitem[Ghi05]{Ghi05}
Paolo Ghiggini.
\newblock Strongly fillable contact 3-manifolds without {S}tein fillings.
\newblock {\em Geom. Topol.}, 9:1677--1687, 2005.

\bibitem[Gir02]{Gir}
Emmanuel Giroux.
\newblock G\'{e}om\'{e}trie de contact: de la dimension trois vers les
  dimensions sup\'{e}rieures.
\newblock In {\em Proceedings of the {I}nternational {C}ongress of
  {M}athematicians, {V}ol. {II} ({B}eijing, 2002)}, pages 405--414. Higher Ed.
  Press, Beijing, 2002.

\bibitem[Gir20a]{Gir20}
Fabio Gironella.
\newblock On some examples and constructions of contact manifolds.
\newblock {\em Math. Ann.}, 376(3-4):957--1008, 2020.

\bibitem[Gir20b]{GirouxIdealLiouvDom}
Emmanuel Giroux.
\newblock Ideal {L}iouville domains, a cool gadget.
\newblock {\em J. Symplectic Geom.}, 18(3):769--790, 2020.

\bibitem[GKZ23]{GKZ}
Hansj\"{o}rg Geiges, Myeonggi Kwon, and Kai Zehmisch.
\newblock Diffeomorphism type of symplectic fillings of unit cotangent bundles.
\newblock {\em J. Topol. Anal.}, 15(3):683--705, 2023.

\bibitem[GNE22]{GhiNie22}
Paolo Ghiggini and Klaus Niederkr\"{u}ger-Eid.
\newblock On the symplectic fillings of standard real projective spaces.
\newblock {\em J. Fixed Point Theory Appl.}, 24(2):Paper No. 37, 18, 2022.

\bibitem[GNW16]{GNW}
Paolo Ghiggini, Klaus Niederkr\"{u}ger, and Chris Wendl.
\newblock Subcritical contact surgeries and the topology of symplectic
  fillings.
\newblock {\em J. \'{E}c. polytech. Math.}, 3:163--208, 2016.

\bibitem[Gro85]{Gro85}
Mikhail Gromov.
\newblock Pseudo holomorphic curves in symplectic manifolds.
\newblock {\em Invent. Math.}, 82(2):307--347, 1985.

\bibitem[Gut14]{Gutt}
Jean Gutt.
\newblock Generalized {Conley}-{Zehnder} index.
\newblock {\em Ann. Fac. Sci. Toulouse, Math. (6)}, 23(4):907--932, 2014.

\bibitem[GZ21]{gironella2021exact}
Fabio Gironella and Zhengyi Zhou.
\newblock Exact orbifold fillings of contact manifolds.
\newblock {\em arXiv preprint arXiv:2108.12247}, 2021.

\bibitem[Hir94]{HirBook}
Morris~W. Hirsch.
\newblock {\em Differential topology}, volume~33 of {\em Graduate Texts in
  Mathematics}.
\newblock Springer-Verlag, New York, 1994.
\newblock Corrected reprint of the 1976 original.

\bibitem[KM63]{zbMATH03188186}
Michel~A. Kervaire and John~W. Milnor.
\newblock Groups of homotopy spheres. {I}.
\newblock {\em Ann. Math. (2)}, 77:504--537, 1963.

\bibitem[Kre99]{Kreck}
Matthias Kreck.
\newblock Surgery and duality.
\newblock {\em Ann. of Math. (2)}, 149(3):707--754, 1999.

\bibitem[KvK16]{KvK16}
Myeonggi Kwon and Otto van Koert.
\newblock Brieskorn manifolds in contact topology.
\newblock {\em Bull. Lond. Math. Soc.}, 48(2):173--241, 2016.

\bibitem[Laz20a]{Laz20}
Oleg Lazarev.
\newblock Contact manifolds with flexible fillings.
\newblock {\em Geom. Funct. Anal.}, 30(1):188--254, 2020.

\bibitem[Laz20b]{MR4100126}
Oleg Lazarev.
\newblock Maximal contact and symplectic structures.
\newblock {\em J. Topol.}, 13(3):1058--1083, 2020.

\bibitem[LMN19]{LMN}
Samuel Lisi, Aleksandra Marinkovi\'{c}, and Klaus Niederkr\"{u}ger.
\newblock On properties of {B}ourgeois contact structures.
\newblock {\em Algebr. Geom. Topol.}, 19(7):3409--3451, 2019.

\bibitem[McD91]{McD91}
Dusa McDuff.
\newblock Symplectic manifolds with contact type boundaries.
\newblock {\em Invent. Math.}, 103(3):651--671, 1991.

\bibitem[McL16]{Reeb}
Mark McLean.
\newblock Reeb orbits and the minimal discrepancy of an isolated singularity.
\newblock {\em Invent. Math.}, 204(2):505--594, 2016.

\bibitem[MNW13]{MNW}
Patrick Massot, Klaus Niederkr\"{u}ger, and Chris Wendl.
\newblock Weak and strong fillability of higher dimensional contact manifolds.
\newblock {\em Invent. Math.}, 192(2):287--373, 2013.

\bibitem[MR23]{McKay}
Mark McLean and Alexander~F. Ritter.
\newblock The {McKay} correspondence for isolated singularities via {Floer}
  theory.
\newblock {\em J. Differ. Geom.}, 124(1):113--168, 2023.

\bibitem[MS18]{zbMATH06864340}
Emmy Murphy and Kyler Siegel.
\newblock Subflexible symplectic manifolds.
\newblock {\em Geom. Topol.}, 22(4):2367--2401, 2018.

\bibitem[{Mur}12]{MurLoose}
Emmy {Murphy}.
\newblock {Loose Legendrian embeddings in high dimensional contact manifolds}.
\newblock {\em arXiv e-prints}, page arXiv:1201.2245, January 2012.

\bibitem[MZ23]{MZ}
Agustin Moreno and Zhengyi Zhou.
\newblock {RSFT} functors for strong cobordisms and applications.
\newblock {\em arXiv preprint arXiv:2308.00370}, 2023.

\bibitem[MZ25]{RSFT}
Agustin Moreno and Zhengyi Zhou.
\newblock A landscape of contact manifolds via rational {SFT}.
\newblock {\em Geom. Topol.}, 29(7):3465--3565, 2025.

\bibitem[Nie06]{Nie06}
Klaus Niederkr\"{u}ger.
\newblock The plastikstufe---a generalization of the overtwisted disk to higher
  dimensions.
\newblock {\em Algebr. Geom. Topol.}, 6:2473--2508, 2006.

\bibitem[NW11]{NieWen11}
Klaus Niederkr\"{u}ger and Chris Wendl.
\newblock Weak symplectic fillings and holomorphic curves.
\newblock {\em Ann. Sci. \'{E}c. Norm. Sup\'{e}r. (4)}, 44(5):801--853, 2011.

\bibitem[OV12]{OV}
Alexandru Oancea and Claude Viterbo.
\newblock On the topology of fillings of contact manifolds and applications.
\newblock {\em Comment. Math. Helv.}, 87(1):41--69, 2012.

\bibitem[Par16]{pardon}
John Pardon.
\newblock An algebraic approach to virtual fundamental cycles on moduli spaces
  of pseudo-holomorphic curves.
\newblock {\em Geom. Topol.}, 20(2):779--1034, 2016.

\bibitem[Par19]{pardon2019contact}
John Pardon.
\newblock Contact homology and virtual fundamental cycles.
\newblock {\em J. Amer. Math. Soc.}, 32(3):825--919, 2019.

\bibitem[RS93]{RS}
Joel Robbin and Dietmar Salamon.
\newblock The {Maslov} index for paths.
\newblock {\em Topology}, 32(4):827--844, 1993.

\bibitem[Sei03]{Sei03}
Paul Seidel.
\newblock A long exact sequence for symplectic {F}loer cohomology.
\newblock {\em Topology}, 42(5):1003--1063, 2003.

\bibitem[Sei08]{biased}
Paul Seidel.
\newblock A biased view of symplectic cohomology.
\newblock In {\em Current developments in mathematics, 2006}, pages 211--253.
  Somerville, MA: International Press, 2008.

\bibitem[Ust99]{U99}
Ilya Ustilovsky.
\newblock Infinitely many contact structures on {$S^{4m+1}$}.
\newblock {\em Internat. Math. Res. Notices}, (14):781--791, 1999.

\bibitem[Ven18]{zbMATH06864027}
Sara Venkatesh.
\newblock Rabinowitz {Floer} homology and mirror symmetry.
\newblock {\em J. Topol.}, 11(1):144--179, 2018.

\bibitem[vK08]{vKo08}
Otto van Koert.
\newblock Contact homology of {B}rieskorn manifolds.
\newblock {\em Forum Math.}, 20(2):317--339, 2008.

\bibitem[Wan15]{MR3418529}
Andy Wand.
\newblock Tightness is preserved by {L}egendrian surgery.
\newblock {\em Ann. of Math. (2)}, 182(2):723--738, 2015.

\bibitem[Yau04]{Yau}
Mei-Lin Yau.
\newblock Cylindrical contact homology of subcritical {S}tein-fillable contact
  manifolds.
\newblock {\em Geom. Topol.}, 8:1243--1280, 2004.

\bibitem[Zho21a]{Zhou_2021}
Zhengyi Zhou.
\newblock {{\((\mathbb{R}\mathbb{P}^{2n-1},\xi_{\mathrm{std}})\)}} is not
  exactly fillable for {{\(n\neq 2^k\)}}.
\newblock {\em Geom. Topol.}, 25(6):3013--3052, 2021.

\bibitem[Zho21b]{Zho21}
Zhengyi Zhou.
\newblock Symplectic fillings of asymptotically dynamically convex manifolds
  {I}.
\newblock {\em J. Topol.}, 14(1):112--182, 2021.

\bibitem[Zho22]{zbMATH07567794}
Zhengyi Zhou.
\newblock Symplectic fillings of asymptotically dynamically convex manifolds
  {II}-{{\(k\)}}-dilations.
\newblock {\em Adv. Math.}, 406:62, 2022.
\newblock Id/No 108522.

\bibitem[Zho23a]{zhou2023contact}
Zhengyi Zhou.
\newblock Contact {$(+ 1)$}-surgeries and algebraic overtwistedness.
\newblock {\em arXiv preprint arXiv:2307.12635}, 2023.

\bibitem[Zho23b]{product}
Zhengyi Zhou.
\newblock On fillings of {{\(\partial (V\times \mathbb{D})\)}}.
\newblock {\em Math. Ann.}, 385(3-4):1493--1520, 2023.

\bibitem[Zho23c]{zbMATH07706508}
Zhengyi Zhou.
\newblock On the cohomology ring of symplectic fillings.
\newblock {\em Algebr. Geom. Topol.}, 23(4):1693--1724, 2023.

\bibitem[Zho24]{intersection}
Zhengyi Zhou.
\newblock On the intersection form of fillings.
\newblock {\em J. Lond. Math. Soc., II. Ser.}, 110(3):30, 2024.
\newblock Id/No e12981.

\end{thebibliography}
